\documentclass[oneside, a4paper,11pt,reqno]{amsart}

\usepackage[english]{babel}
\usepackage[colorlinks]{hyperref}
\usepackage{epsfig}
\usepackage{mathtools}
\usepackage{amsmath}

\usepackage{amsthm}

\usepackage{amssymb}
\usepackage{amscd}
\usepackage{tikz}
\usepackage{graphicx}
\usepackage{color}
\usepackage{verbatim}
\usepackage[utf8]{inputenc}
\usepackage{pgf,tikz}
\usepackage{mathrsfs}
\usetikzlibrary{arrows}

\newtheorem{thm}{Theorem}[section]
\newtheorem{lem}[thm]{Lemma}
\newtheorem{prop}[thm]{Proposition}

\newtheorem{hyp}[thm]{Hypothesis}

\newtheorem{rem}[thm]{Remark}

\newcommand{\Pa}{\mathcal P}
\newcommand{\B}{\mathcal B}

\newcommand{\Z}{\mathbb{Z}}
\newcommand{\N}{\mathbb{N}}

\newcommand{\A}{\mathcal{A}}

\newcommand{\f}{\varphi}
\newcommand{\mup}{\overline{\mu}}
\newcommand{\tp}{\overline{T}}
\newcommand{\Mp}{\overline{M}}
\newcommand{\R}{\mathcal{R}}

\newcommand{\e}{\epsilon_0}

\newcommand{\pM}{\hat{M}}
\newcommand{\pt}{\hat{T}}
\newcommand{\pmu}{\hat{\mup}}
\newcommand{\Fb}{\overline F}
\newcommand{\PP}{\mathbb{P}}

\newcommand{\F}{F}

\newcommand{\Ng}{\textbf{N}}

\newcommand{\ls}{\underline{\textbf{l}}}

\newcommand{\qs}{\underline{q}}

\title{Averaging theorems for slow fast systems in $\Z$-extensions (discrete time)}
\address{Univ Brest,  CNRS  UMR 6205, Laboratoire de Mathématiques 
de Bretagne Atlantique.}
\author{Maxence Phalempin}

\begin{document}

\maketitle

\begin{abstract}
We study the averaging method for flows perturbed by a dynamical system preserving an infinite measure. Motivated by the case of perturbation by the collision dynamic on the finite horizon $\mathbb Z$-periodic Lorentz gas  and in view of future development, we establish our results in a general context of perturbation by $\mathbb Z$-extension over chaotic probability preserving dynamical systems.
As a by product, we prove limit theorems for non-stationary Birkhoff sums for such infinite measure preserving dynamical systems.
\end{abstract}

\section{introduction}
We call perturbed differential equation or equivalently slow fast system the following Cauchy problem

\begin{align}\label{pertx}
\frac{dx^{\epsilon}_t(x,\omega)}{dt}=F'(x^{\epsilon}_t(x,\omega),T^{\lfloor \frac t \epsilon \rfloor}(\omega)) ,\quad x_0^\epsilon(x,\omega)=x ,\quad \forall (t,x,\omega)\in [0,+\infty)\times\mathbb R^d\times M\, ,
\end{align}
where $ F' :\mathbb R^d \times M \rightarrow \mathbb R^d$ is a function, 
Lipschitz in its first coordinate (in $\mathbb R^d$) and where $T:M\rightarrow M$ is a  measurable map on a measurable space $M$ and preserving a measure $\mu$ (i.e. $(M,T,\mu)$ is a measure preserving dynamical system). 
In~\eqref{pertx}, the term $T^{\lfloor t/\epsilon \rfloor}$ represents the fast motion when $\epsilon \rightarrow 0^+$,
 whereas $x_t^\epsilon$ is considered as a slow motion. The behavior of $(x_t^\epsilon)_{t\geq 0}$ has been studied in many cases when the map $T$ preserves a probability measure $\mu$ and the resulting 
dynamical system has rapid mixing properties, we can notably cite the works done for exponentially mixing flows and transformations as Anosov flows or Billiard maps or ergodic toral automorphisms  by Anosov \cite{anosovslowfast}, Arnold \cite{arnoldequadiff}, Khasminskii \cite{kasminski}, Kifer \cite{kifer} and Pène \cite{peneESAIMPS}, Dolgopyat\cite{Dolgopyat} Chevyrev, Friz, Korepanov and Melbourne in \cite{korepanovslowfast}, among others. 
In these settings they proved that while the dynamic accelerates ($\epsilon \rightarrow0^+$), the perturbed solution $(x_t^\epsilon)_{t\geq 0}$ converges to some averaged motion $(w_t)_{t\geq 0}$ solution of an ordinary averaged differential equations  
\begin{align}\label{equadiffw}
\frac{dw_t(x
)}{dt}=\overline F(w_t(x
)) ,\quad w_0^\epsilon(x)=x ,\quad \forall (t,x)\in [0,+\infty)\times\mathbb R^d,
\end{align}
where $\overline F(x):=\int_{M} F'(x,\cdot)d\mu$. Actually they go further proving limit theorems for the error term $(e_t^\epsilon)_{t\geq 0}$ between the slow motion $(x_t^\epsilon)_{t\geq 0}$ and the averaged solution $(w_t)_{t\geq 0}$. 
In the present article, our motivation  is to study slow fast systems with fast dynamics $T^{\lfloor t/\epsilon\rfloor}$ preserving an infinite measure $\mu$. 
To our knowledge, this article contains the  first averaging results for perturbation by infinite measure preserving dynamical systems.  
In this context, the previous definition of the averaged function $\overline F$ either
doesn't make sense in general 
(indeed it makes sense only if $F'(x,\cdot)$
is $\mu$-integrable, but then $F'(x,\cdot)-\overline F(x)$ is not $\mu$-integrable).
Thus, the classical technique consisting in decomposing the $F'$ in the sum of a drift $\overline F$ and a centered part $F'-\overline F$ in order to study the oscillations of the latter to obtain the speed of convergence and a limit theorem for the error term  $(e_t^\epsilon)_{t\geq 0}$ needs to be adapted.
At least two settings naturally arise
with two radically different behaviours. 
A first one is to consider some case where $F'$ is integrable which is investigated in \cite[Chapter 8]{MPphd} (where a limit theorem is established, without averaging). 
A second one, that we investigate here, is to consider a differential equation~\eqref{pertx} 
obtained by a perturbation of a differential equation~\eqref{equadiffw}, by considering
a map $F'$
given by a sum of a null-integral function 
$F$ and of the drift $\overline F$ appearing in~\eqref{equadiffw} (note the $\overline F$ is independent of the dynamics). In this context, 
we consider the solution $(x_t^\epsilon)_t$ of~\eqref{pertx} and the solution $(w_t)_t$
of~\eqref{equadiffw} and the error $e_t^\epsilon$
defined by
\[
e_t^\epsilon(x,\omega):=x_t^\epsilon(x,\omega)-w_t(x),\quad\forall (t,x,\omega)\in[0,+\infty)\mathbb R^d\times M\, .
\]
for all $T'>0$, 
\begin{thm}[See Theorem \ref{thmequadif} for a more general statement]\label{THM1}
Let $\epsilon >0$, and $(M,T,\mu)$ be the collision dynamics associated to the $\Z$-periodic Lorentz gas with finite horizon (see Section~\ref{secmodels}) and suppose $F'(x,\omega)=\overline F(x)+F(x,\omega)$
with $\overline F : \mathbb R^d \rightarrow \mathbb R^d$ and 
$F: \mathbb R^d \times M \rightarrow \mathbb R^d$ such that $\int_M F(x,\cdot)d\mu=0$ for all $x\in\mathbb R^d$ and $F(\cdot,\omega)$ is uniformly  $C^2_b$ (uniformly bounded, and with uniformly bounded first and second order derivatives) 
and $F(x,\cdot)$ is H\"older continuous (with respect to the euclidean metric, or dynamically H\"older) and sharply vanishing at infinity.\footnote{The precise assumptions on $F$ are stated in Theorem~\ref{thmbirkhoff} where $\Vert\cdot\Vert_{lip}$ can be replaced by the 
H\"older norm.}
Then, for any $x \in \mathbb R^d$, the error term $(e_t^\epsilon(x,\cdot))_{t\in [0,T']}$ satisfies a limit theorem for the strong convergence in law:
\begin{align*}
(\epsilon^{-3/4}e_t^\epsilon(x,\cdot))_{t \in [0,T']}\underset{\epsilon \rightarrow 0}{\overset{\mathcal{L}_\mu,\|.\|_{\infty}}{\rightarrow}}(y_t(x))_{t \in [0,T']},
\end{align*}
where $\mathcal{L}_\mu$ denotes the strong convergence in law\footnote{This means the convergence in distribution with respect to every probability measure absolutely continuous with respect to $\mu$.} with respect to the measure $\mu$ and where $(y_t(x))_{t\geq 0}$ is the solution of the following stochastic differential equation
$$
dy_t(x)=\sqrt{a(w_s(x))}dB_{L'_s(0)}+D\overline F(w_s(x))y_sds,
$$
with $(B_t)_{t\geq 0}$ a multi-dimensional standard Brownian motion i.e a stochastic process of which the coordinates $B^{(i)}$ are independent uni-dimensional Brownian motions (B.m) of variance $1$.  The process $(L_s')_{s \geq 0}$ is a local time associated to a Brownian motion $(B_t')_{t \geq 0}$ 
independent of the previous process whereas $a: \mathbb{R}^d \rightarrow \mathbb{R}^{d^2}$ is a map such that for any $x\in \mathbb R^d$, $a(x)$ is a non negative symmetric matrix that can be interpreted as a local asymptotic variance in $x$ (see~\eqref{asymptvar2}).
\end{thm}
This result relies on the following key statement which is interesting in itself since it provides the stochastic convergence for perturbed ergodic sums and in particular implies the central limit theorem on $\Z$-periodic Lorentz gas as stated in \cite{PT20} (for a more general class of observables than the one considered in~\cite{PT20} which were
locally constant) by considering the case when  $F(x,\omega)=F(\omega)$ does not depends on its first coordinate.
 \begin{thm}[see Theorem \ref{thmbirkhoff} page \pageref{thmbirkhoff}]\label{THM2}
 Under the same assumptions on the map $F$ 
as in the above Theorem~\ref{THM1}, 
for any $t\in [0,+\infty)$, we denote, for any $(x,\omega) \in\mathbb R^d\times M$,  by $(v_t^\epsilon(x,\omega))_{t\geq 0}$ the \textbf{perturbed Birkhoff sum} defined as follows
$$
v_t^\epsilon(x,\omega):=\epsilon^{1/4}\int_0^{t/\epsilon} F(w_{\epsilon s}(x),T^{\lfloor s \rfloor}\omega)\, ds\, .
$$
Then the process $(v_t^\epsilon(x,\cdot))_{t\geq 0}$  satisfies the following functional central limit theorem for any $T \geq 0$ and $x \in \mathbb{R}^d$,
\begin{align*}
\left(v_t^{\epsilon}(x,\cdot)\right)_{t \in [0,T]} \underset{\epsilon \rightarrow 0}{\overset{\mathcal{L}_{\mu}, \|.\|_\infty}{\rightarrow}} \left(\int_0^t \sqrt{a(w_s(x))}dB_{L_s'(0)}\right)_{t \in [0,T]}\, ,
\end{align*}

with $(B_t)_{t\geq 0}$ and $(L'_t)_{t\geq 0}$ the same processes as in the previous theorem.
\end{thm}
The present article is organized as follows. The first section is dedicated to the detailed example of the collision dynamics associated to the Lorentz gas. Section \ref{secmodels} presents the general settings and the results under abstract hypotheses.
Section~\ref{secZext} contains the proof of the limit theorem
for perturbed Birkhoff sums (proof of Theorem~\ref{thmbirkhoff} which generalizes  Theorem~\ref{THM2}) and section~\ref{secbirk} the averaging result for transformations (proof of Theorem~\ref{thmequadif} which generalizes
Theorem~\ref{THM1}).

\section{Collision dynamics associated to $\Z$-periodic Lorentz gas with finite horizon.}\label{secmodels}

We present here the Billard transformation associated to the collision dynamics of the $\Z$-periodic Lorentz gas $(M,T,\mu)$ (this dynamical system is also called the 
discrete time $\Z$-periodic Lorentz gas). 
Consider a domain  $\R_0$ corresponding to the infinite flat cylinder $\mathbb R\times\mathbb T$ (where $\mathbb T:= \mathbb{R}/\Z$)  doted of open convex obstacles belonging to a periodic, locally finite, family $\{O_m+(l,0), l\in \Z, m\in I\}$  with $C^3$ boundary, positive curvature and with pairwise disjoint closures, placed along the cylinder. So
$$
\R_0:=(\mathbb T\times \mathbb{R})\backslash \bigcup_{m,l}(O_m+(l,0)).
$$
We assume furthermore that the horizon is finite horizon, i.e that $\R_0$ contains no line.
The dynamical system $(M,T,\mu)$ describes the dynamics at collision times of a point particle moving in $\R_0$, going straight inside $\R_0$ and colliding against the obstacles
according to the Snell-Descartes reflection law (the reflected angle is equal to the incident angle).
The phase space $M$ is the set of post-collisional unit vectors, i.e. the set of couples of (position,direction of its speed) based on $\partial \R_0$ and pointing into $\R_0$. Formally, $M$ is the set of $(q,v)\in\partial\R_0\times\mathbb S^1$
(where $\mathbb{S}^1$ is the unit circle) such that $\langle v, n(q) \rangle \geq 0$ where $n(q)$ denotes the unit normal vector to $\partial \R_0$ in $q$ directing into $\R_0$.


\begin{figure}[!htbp] 
\begin{center}
\definecolor{ccqqqq}{rgb}{0.8,0.,0.}
\definecolor{qqzzqq}{rgb}{0.,0.6,0.}
\definecolor{ffqqqq}{rgb}{1.,0.,0.}
\begin{tikzpicture}[line cap=round,line join=round,>=triangle 45,x=0.5026463381935778cm,y=0.31089064543404965cm]
\clip(-15.209033729977385,-5.114520111870424) rectangle (6.05059140582434,5.043519325003994);
\draw [line width=2.pt] (0.,0.) ellipse (1.0052926763871557cm and 0.6217812908680993cm);
\draw [shift={(-5.,5.)},line width=2.pt]  plot[domain=-1.5707963267948966:0.,variable=\t]({1.*3.8*cos(\t r)+0.*3.8*sin(\t r)},{0.*3.8*cos(\t r)+1.*3.8*sin(\t r)});
\draw [shift={(-5.,-5.)},line width=2.pt]  plot[domain=0.:1.5707963267948966,variable=\t]({1.*3.8*cos(\t r)+0.*3.8*sin(\t r)},{0.*3.8*cos(\t r)+1.*3.8*sin(\t r)});
\draw [shift={(5.,-5.)},line width=2.pt]  plot[domain=1.5707963267948966:3.141592653589793,variable=\t]({1.*3.8*cos(\t r)+0.*3.8*sin(\t r)},{0.*3.8*cos(\t r)+1.*3.8*sin(\t r)});
\draw [shift={(5.,5.)},line width=2.pt]  plot[domain=3.141592653589793:4.71238898038469,variable=\t]({1.*3.8*cos(\t r)+0.*3.8*sin(\t r)},{0.*3.8*cos(\t r)+1.*3.8*sin(\t r)});
\draw [shift={(-15.,5.)},line width=2.pt]  plot[domain=-1.5707963267948966:0.,variable=\t]({1.*3.8*cos(\t r)+0.*3.8*sin(\t r)},{0.*3.8*cos(\t r)+1.*3.8*sin(\t r)});
\draw [shift={(-15.,-5.)},line width=2.pt]  plot[domain=0.:1.5707963267948966,variable=\t]({1.*3.8*cos(\t r)+0.*3.8*sin(\t r)},{0.*3.8*cos(\t r)+1.*3.8*sin(\t r)});
\draw [shift={(-5.,-5.)},line width=2.pt]  plot[domain=1.5707963267948966:3.141592653589793,variable=\t]({1.*3.8*cos(\t r)+0.*3.8*sin(\t r)},{0.*3.8*cos(\t r)+1.*3.8*sin(\t r)});
\draw [shift={(-5.,5.)},line width=2.pt]  plot[domain=3.141592653589793:4.71238898038469,variable=\t]({1.*3.8*cos(\t r)+0.*3.8*sin(\t r)},{0.*3.8*cos(\t r)+1.*3.8*sin(\t r)});
\draw [shift={(5.,5.)},line width=2.pt]  plot[domain=-1.5707963267948966:0.,variable=\t]({1.*3.8*cos(\t r)+0.*3.8*sin(\t r)},{0.*3.8*cos(\t r)+1.*3.8*sin(\t r)});
\draw [shift={(5.,-5.)},line width=2.pt]  plot[domain=0.:1.5707963267948966,variable=\t]({1.*3.8*cos(\t r)+0.*3.8*sin(\t r)},{0.*3.8*cos(\t r)+1.*3.8*sin(\t r)});
\draw [shift={(15.,-5.)},line width=2.pt]  plot[domain=1.5707963267948966:3.141592653589793,variable=\t]({1.*3.8*cos(\t r)+0.*3.8*sin(\t r)},{0.*3.8*cos(\t r)+1.*3.8*sin(\t r)});
\draw [shift={(15.,5.)},line width=2.pt]  plot[domain=3.141592653589793:4.71238898038469,variable=\t]({1.*3.8*cos(\t r)+0.*3.8*sin(\t r)},{0.*3.8*cos(\t r)+1.*3.8*sin(\t r)});
\draw [line width=2.pt] (10.,0.) ellipse (1.0052926763871557cm and 0.6217812908680993cm);
\draw [line width=2.pt] (-10.,0.) ellipse (1.0052926763871557cm and 0.6217812908680993cm);
\draw [line width=2.pt,color=ffqqqq] (-15.,5.)-- (15.,5.);
\draw [line width=2.pt,color=ffqqqq] (-15.,-5.)-- (15.,-5.);
\draw (-1.7137896049836079,0.5913267626151101) node[anchor=north west] {$O_1+0$};
\draw (-10.91139355019933,0.6444153856755176) node[anchor=north west] {$O_1-1$};
\draw (8.160694284252319,0.471877360729193) node[anchor=north west] {$O_1+1$};
\draw (3.648161324117621,4.320802532608742) node[anchor=north west] {$O_2+1$};
\draw (3.6747056356478254,-3.1779654746738273) node[anchor=north west] {$O_2+1$};
\draw (12.460872752145384,4.45352409025976) node[anchor=north west] {$O_2+2$};
\draw (12.447600596380282,-3.3106870323248465) node[anchor=north west] {$O_2+2$};
\draw (-6.7041201726619795,4.5464291806154735) node[anchor=north west] {$O_2+0$};
\draw (-6.584670770776061,-3.2575984092644386) node[anchor=north west] {$O_2+0$};
\draw (-14.109983089588924,4.413707622964455) node[anchor=north west] {$O_2-1$};
\draw (-14.388698360656068,-3.2708705650295404) node[anchor=north west] {$O_2-1$};
\draw (-8.787848627783001,2.9006818657428393) node[anchor=north west] {$x$};
\draw [line width=2.pt,color=qqzzqq] (-8.306144483740791,1.0634159534435474)-- (-5.226459494154057,-1.2067538838736698);
\draw [line width=2.pt,color=ccqqqq] (-5.226459494154057,-1.2067538838736698)-- (-4.077083771656735,-0.7135464381838353);
\draw [line width=2.pt,color=ccqqqq] (-4.077083771656735,-0.7135464381838353)-- (-4.260736560024999,-0.7424254302813408);
\draw [line width=2.pt,color=ccqqqq] (-4.077083771656735,-0.7135464381838353)-- (-4.214720008370887,-0.8322138237527775);
\draw [line width=2.pt,color=ccqqqq] (-8.306144483740791,1.0634159534435474)-- (-7.333706969194213,0.3465899398532405);
\draw [line width=2.pt,color=ccqqqq] (-7.333706969194213,0.3465899398532405)-- (-7.519660452465569,0.6325111396647947);
\draw [line width=2.pt,color=ccqqqq] (-7.333706969194213,0.3465899398532405)-- (-7.686994483514875,0.4472484624316367);
\draw (-3.9700560850509565,1.215118083574899) node[anchor=north west] {$T(x)$};
\end{tikzpicture}
  \caption{\footnotesize{Illustration of the map $T$ corresponding to the dynamics of the discrete time $\Z$-periodic Lorentz gas in finite horizon, with two periodic shape of obstacles.}} \label{fig:GG08}
    \end{center}
      \vskip -0.5cm
\end{figure}
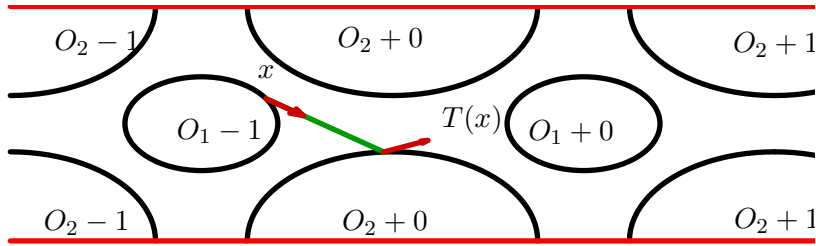

The transformation 
$T$ maps a couple (position/direction of the momentum) to the new couple (position/direction of the momentum) after one collision against an obstacle. This transformation $T$ preserves an infinite 
measure $\mu$  absolutely continuous with respect to the Lebesgue measure with an explicit $\mathbb Z$-periodic density (as recalled in~\cite[Section 2.1]{maxautointer})\\
The fact that the horizon is finite ensures that the free flight between two collisions is uniformly bounded. 
A crucial fact is that the dynamical system $(M,T,\mu)$
can be represented as a $\Z$-extension with centered bounded step function $\phi : \Mp \rightarrow \Z$ over the probabilistic dynamical system $(\Mp,\tp,\mup)$ called the Sinai Billiard. The Sinai billiard $(\Mp,\tp,\mup)$ is the quotient modulo $\Z$ of the previous system $(M,T,\mu)$. Thus $\tp$ is the billiard transformation with obstacles $\bar O_i$ on the two-dimensional flat torus $\mathbb T^2:=\mathbb R^2\times\mathbb Z^2$ (with $\bar O_i$ the image of $O_i$ by the canonical quotient map). The ergodicity and mixing of this Sinai Billiard $(\Mp,\tp,\mup)$ has been proved by Sinai in \cite{Sinai}, exponential mixing has been proved by Young via the construction of towers in \cite{young}.
The stochastic properties of the Lorentz gas $(M,T,\mu)$ have been studied in previous decades  using the chaotic properties of $(\Mp,\tp,\mup)$. This system was  proved to be recurrent in \cite{schmidt-z-rec}, ergodic (conservatively ergodic) in \cite{Sim89ergobilliard,penergodic} and to satisfy a central limit theorem for some class of observables in \cite{PT20}.
Given  some $\beta\in(0;1)$,
we endow the dynamical system $(\Mp,\tp,\mup)$ with the dynamical metric $d=d_\beta$ given by
\begin{align}\label{dynamicmetric}
\forall x, y \in \Mp,\quad
    d(x,y):=\beta^{s_0(x,y)},
\end{align}
 where $s_0: \Mp \times \Mp \rightarrow \N$ is the separation time such that for $x,y\in \Mp$, $s_0(x,y)\geq n$ iff 
$x$ and $y$ are in the same
continuity component of $T^i$ for any $i\in \{-n,\dots,n\}$. 
For any $\eta\in(0,1)$, there exists $\beta\in(0,1)$ such that
the class of Lipschitz functions with respect to $d_\beta$  comprises the class of $\eta$-H\"older functions for the usual euclidean metric on $\Mp$ (see \cite[Exercise 5.49 page 124]{chernov}). 
We also introduce, for $k,n \in \Z$, the partitions $\xi_k^n$  as of $\Mp$ such that for any $A\in \xi_{k}^n$ two elements $x,y\in A$ if their iterates $T^i x$ and $T^i y$ share the same obstacle ($\partial O_{m_i}\times \mathbb S$ for some $m_i\in I$) for any $i\in \{k,\dots,n\}$. In particular if $k=n$ and $A\in \xi_k^n$,
$$
x,y\in A \Rightarrow s_0(x,y)\geq n.
$$
Thus these sets are related to the 
the separation time $s_0$, which dominates the separation time in the "coding" of the Young Towers on which our main theorem are stated, Indeed maps $f:\Mp\rightarrow \mathbb R ^d$ that are constant on elements of $\xi_{-n}^n$ can be lifted to maps onto a Young Towers structure  (see \cite{young} or 
\cite[Section 4]{maxautointer}).
Other stochastic properties have already been stated for the Sinai billiard system $(\Mp,\tp,\mup)$ such as Central limit theorem for smooth observables (including the step function $\Phi$) by Bunimovich, Sinai and Chernov (see \cite{sinaiclt}, \cite{buni_sinai_chernov}) and the already mentioned works of Kifer, Khasminiskii and Pène on perturbed differential equations (see \cite{kasminski}, \cite{kifer}, \cite{phdpene}) among others.

\section{General results for $\mathbb Z$-extensions}
\label{secZext}
In this section, we introduce a general context of $\mathbb Z$-extension and establish general results in this context.
The fact that the $\mathbb Z$-periodic Lorentz gas in finite horizon introduced in Section~\ref{secmodels} satisfies our general assumptions has been proved in 
\cite[Section 4]{maxautointer} and will be recalled in remark \ref{remhyp} below.

\subsection{General model of $\Z$-extension.}

In this paper we consider a infinite measure preserving dynamical system $(M,T,\mu)$ described as a $\Z$-extension of an ergodic probabilistic dynamical system $(\Mp,\tp,\mup)$ with a centered bounded step function $\phi :\Mp \rightarrow \Z$. The dynamics on such a system is described by the following skew product for any $(\omega,m)\in \Mp \times \Z$ by :
\begin{align*}
    T(\omega,m)=(\tp (\omega),m+\phi(\omega)).
\end{align*}
When iterated, the dynamic brings the following identity
\begin{align*}
    T^n(\omega,m)=(\tp^n (\omega),m+S_n\phi(\omega)),
\end{align*}
where $S_n\phi:=\sum_{k=0}^{n-1}\phi\circ\tp^k$. 
We endow $(M,T)$ with the $T$-invariant measure $\mu$ defined on the $\Z$-extension $\Mp\times \Z$ is a product measure $\mu:=\mup \otimes m$ where $m$ is the counting measure on $\Z$. 
The assumption that $\phi$ is centered (i.e $\int_{\Mp} \phi d\mup=0$) combined with the ergodicity of $(\Mp,\tp,\mup)$ ensures that the Birkhoff sum $S_n\phi$ is recurrent as a random walk on $\Z$. This implies that 
$(M,T,\mu)$ is a recurrent dynamical system (see \cite{schmidt-z-rec}) then additional properties such as mixing (which will be implied by our abstract settings \ref{H2}) will ensure the ergodic conservativity of such dynamical systems.

The main result here is given under some abstract hypotheses : The quotient probability invariant dynamical system $(\Mp,\tp,\mup)$ satisfies the following hypotheses \ref{H2} in the sense that it can be modeled to an exponentially  mixing Young tower  (described in hypothesis \ref{H1}).

\begin{hyp}
\label{H1}
Let $(\pM,\pt,\pmu)$ be a probability preserving dynamical system and an application $\hat \phi: \pM \rightarrow \Z$. There exists $\beta\in]0,1[$ and $\varepsilon>0$ and a
complex Banach space $\B=\B_{\beta,\epsilon}$  
such that
\begin{enumerate}
\item[(0)] the constant function $1_{\pM}$ is in $\B$ and 
$\B$ is continuously included in $L^1_{\pmu}$.
\item[(1)] there is a partition $(\Delta_{i,k})_{i,k \in \N}$ of $\pM$ and a splitting time $s: \pM \times \pM \rightarrow \N$ such that for any $i,k \in \N$, and any $x,y \in \Delta_{i,k}$, $s(x,y)\geq 1$. In addition for any $x,y \in \pM$ such that $s(x,y)\geq 1$, $s(\pt x, \pt y)=s(x,y)-1$. 
\item[(2)] 
the Banach space $(\B,\|.\|_{\B})$ is described as follows:
 a function $f : \pM \rightarrow \mathbb R$ is in $\B$ if
$$
\|f\|_\B=\|f\|_\infty'+\|f\|_{lip}'<\infty
$$
with 
$$
\|f\|_{lip}'=\sup_{l,k \in \N} e^{-\epsilon l} \|f\|_{lip,\Delta_{l,k}} \textit{    et   }\|f\|_\infty'=\sup_{l,k \in \N}e^{-\epsilon l}\|f\|_{\infty,\Delta_{l,k}}.
$$
where $\|f\|_{lip,\Delta_{l,k}}:=\sup_{x,y \in \Delta_{l,k}}\frac{|f(x)-f(y)|}{\beta^{s(x,y)}}$ and $\|f\|_{\infty,\Delta_{l,k}}:=\|f_{\Delta_{l,k}}\|_{\infty}$.\\


We will denote $l(x)=l$ if there is $k\in \N$ such that $x \in \Delta_{l,k}$.
\item[(3)] the transfer operator $P$ of $(\pM,\pt,\pmu)$ satisfies for $f\in \B$ 
$$
Pf(x)=\sum_{w \in \pt^{-1}(\{x\})} e^{-h(w)}f(w)
$$
with $h\in \B$ a function Lipschitz with respect to $\beta^{s(\cdot,\cdot)}$, i.e.
$$\|h\|_{lip}'=\sup_{i,k \in \N}\sup_{x,y \in \Delta_{i,k}}\frac{|h(x)-h(y)|}{\beta^{s(x,y)}}$$
is a finite quantity. Furthermore, for any $i,k\in \N$ and $x,y \in \Delta_{i,k}$  there is a bijection $X_0 :\pt^{-1}(\{x\})\rightarrow \pt^{-1}(\{y\})$ preserving the partition of $\hat M$, i.e.  $s(z,X_0(z))=1+s(x,y)$.
\item[(4)] For any $x,y \in \pM$ such that $s(x,y)>1$, $\hat{\phi}(x)=\hat{\phi}(y)$.  Denote $S_n:=\sum_{k=0}^{n-1}\hat \phi \circ \pt^k$. There is
$\beta_0>0$ such that or any $b \in ]0,\beta_0[$, there is $\alpha \in ]0,1[$, $C>0$ such that for any $f \in \B$, any positive integer $n$ and any $u\in [-b,b]$, the operator $P_u:=P(e^{iu\hat\phi}\cdot)$ satisfies
\begin{align}\label{eqspectrgap}
P_u^n(f):=P(e^{iu S_n}f)=\lambda_u^n\Pi_u(f)^n+N_u^n(f),    
\end{align}
where
\begin{itemize}
\item $\lambda_. \in C^4([-b,b], \mathbb C)$ satisfies $\lambda_u=1-\Sigma \frac {u^2} 2 +O(u^3)$ with $\Sigma:=\sum_{k\geq 0}E_{\hat{\mup}}( \hat \phi \circ \hat{T}^k \hat \phi)$
\item $u \mapsto  \Pi_u \in C^2([-b,b], L(\B))$  satisfies $\Pi_0(\cdot)=E_{\pmu}(\cdot)1$, in addition, for any $u \in [-b, b]$, $\Pi_u$ is a projection, thus $\Pi_u^n=\Pi_u$.
\item $N_u \in L(\B)$ satisfies $\|N_u^n\|\leq C \alpha^n$.
\end{itemize}
In addition, for any $u \in [-\pi, \pi]\backslash [-b,b]$, $\|P_u^n\| \leq C\alpha^n$.
\end{enumerate}
\end{hyp}

We now state the main hypothesis \ref{H2} which guarantees that $(\Mp,\tp,\mup)$ has a Young tower structure with Lipschitz  observables that can be approximated onto $(\pM,\pt,\pmu)$.
\begin{hyp}\label{H2}
The set $\Mp$ is a metric space with metric $d$. We define as follow the Lipschitz norm of an application $g:\Mp \rightarrow \mathbb{R}$, $\|g\|_{lip}:=\sup_{x,y \in \Mp}\frac{|g(x)-g(y)|}{d(x,y)}$.
We suppose that $(\Mp,\tp,\mup)$ and $(\pM,\pt,\pmu)$ are common factors of a dynamical system $(M_1,T_1,\mu_1)$ with respective projections $\pi_1$ and $\pi_2$. We suppose that there is a map 
 $\hat \phi :\pM \rightarrow \Z$ such that $(\pM,\pt,\pmu)$ and $\hat{\phi}$ satisfy Hypothesis \ref{H1} and $\phi\circ\pi_1=\hat\phi\circ\pi_2$, which implies by induction that
$$
\overline{S_n}\circ \pi_1= S_n \circ \pi_2,
$$
for any $n \in \N$ where $\overline S_n:= \sum_{k=0}^{n-1}\phi\circ T^k$ and where the notation $S_n$ comes from hypothesis \ref{H1}).\\
In addition 
there is $0<\beta_0<1$ and $C>0$ such that for any $n \in \mathbb{N}$, and for any Lipschitz bounded function $g : \Mp \rightarrow \mathbb{R}$ there is,
for any $n \in \N$, a function $g_n$ $\pmu$-centered (in the sense that $E_{\pmu}(g_n(x,.a))=0$) satisfying :
\begin{align}\label{unif}
\|g\circ\tp^n \circ \pi_1-g_n\circ\pi_2 \|_{\infty} \leq C \beta_0^n\|g\|_{lip}
\end{align}

$$
\|g_n\|_{\infty} \leq \|g\|_{\infty}
$$
and such that for any  $x,y \in \pM$, 
\begin{align}\label{constant}
s(x,y) > 2n \Rightarrow g_n(x)=g_n(y).
\end{align}
\end{hyp}

\begin{rem}[Application to the periodic Lorentz gas]\label{remhyp}
We consider here the $\mathbb Z$-periodic Lorentz gas $(M,T,\mu)$ introduced in section~\ref{secmodels}.
The fact that Hypothesis~\ref{H1}  applies
with $(\pM,\pt,\pmu)$ the symbolic factor of Young Tower constructed by Young~\cite{young} for the Billiard transformation has been explained in 
\cite[section 4]{maxautointer}). More precisely, there is 
$\bar\beta\in]0,1[$ such that for any $\beta\in]\bar\beta,1[$, there exists $\bar\epsilon$ such that  Hypothesis~\ref{H1}
holds true for this $\beta$ and 
for any $\epsilon\in]0,\bar\epsilon[$. 
Let $\beta\in]\bar\beta,1[$.
We consider the dynamical metric  $d=d_\beta$ on $\Mp$ given by~\eqref{dynamicmetric}. 
 Let $g:\Mp\rightarrow\mathbb R$ be a Lipschitz function with respect to $d$. For any integer $n\ge 0$, set $$
h_n:=E_{\mup}(g|\xi_{-n}^n).
$$
Then the map $h_n(x,\tp^n(\cdot))$ is constant over the atoms of $\xi_0^{2n}$ and there is a map $g_n(x,\cdot): \pM \rightarrow \mathbb R^d$ such that
$$
h_n(x,\cdot)\circ\tp^n\circ \pi_1=g_n(x,\cdot)\circ \pi_2.
$$
Furthermore this map satisfies 
$$
\|g\circ\tp^n\circ \pi_1-g_n\circ \pi_2\|_{\infty} \leq \beta^n \|g\|_{lip},
$$
since $\sup_{A\in \xi_{-n}^n}\textit{diam}(A)\le \beta^n$, by definition of $d=d_\beta$. 
In addition $\|g_n\|_\infty=\|h_n\|_\infty\leq \|g\|_\infty$ and \eqref{constant} is satisfied.
\end{rem}

\subsection{Perturbed ergodic Birkhoff sum, result for transformation}\label{sec:perturbBirkhoff}

The next statement provides a  limit theorem for perturbed Birkhoff sums $\sum_{k\geq 0}F(s,T^k\cdot)$ for some perturbed observable $F: \mathbb R^d\times M \rightarrow  \mathbb R^d$ on the $\Z$-extension $(M,T,\mu)$ :

\begin{thm}[Convergence of perturbed Birkhoff sums for the map]\label{thmbirkhoff}
Fix $\epsilon_0>0$ and a dynamical system $(M,T,\mu)$ described as a $\Z$-extension over a system $(\Mp,\tp,\mup)$ satisfying Hypothesis \ref{H2} above with bounded step function $\phi :\Mp \rightarrow \Z$. Let $F : \mathbb R^d\times \Mp \times \Z \rightarrow \mathbb R^d$ satisfying
\begin{itemize}
\item 
 $\sup_{x \in \mathbb R^d}\sum_{a \in \Z}\|F(x,.,a)\|_{lip}<\infty$ 
\item  for any $x\in\mathbb R^d$, $F(x,\cdot)$ has zero $\mu$-integral over $M$.
\item $\sup_{x \in \mathbb R^d}\sum_{a \in \Z}|1+a|^{2(1+\e)}\|F(x,.,a)\|_{\infty}<\infty$
\item these functions $F(\cdot,\omega)$ are uniformly lipschitz (uniformly in $\omega \in M$).
\end{itemize}
For $x \in \mathbb R^d$, let $s \mapsto  w_s(x)$ be a Lipschitz function.
Set for any $t\in [0,T]$ and $\omega \in M$, 
\begin{align}\label{eqperturbirkhof}
v_t^\epsilon(x,\omega):=\epsilon^{1/4}\int_0^{t/\epsilon} F(w_{\epsilon s}(x),T^{\lfloor s \rfloor}\omega)ds.    
\end{align}

Then for $T \geq 0$ and $x \in \mathbb R^d$,
\begin{align*}
\left(v_t^{\epsilon}(x,\cdot)\right)_{t \in [0,T]} \underset{\epsilon \rightarrow 0}{\overset{\mathcal{L}_{\mu}, \|.\|_\infty}{\rightarrow}} \left(\int_0^t \sqrt{a(w_s(x))}dB_{L_s'(0)}\right)_{t \in [0,T]}\, ,
\end{align*}
where $\mathcal{L}_{\mu}$ means
\footnote{recall that this means the convergence in distribution with respect to every probability measure absolutely continuous with respect to $\mu$.} 
the strong convergence in distribution  with respect to $\mu$, $(B_t)_{t\geq 0}$ the standard Brownian motion and $(L_s')_{s \geq 0}$ is the local time of a Brownian motion $(B'_t)_{t \geq 0}$ with variance $\Sigma$ (with $\Sigma$ defined as in Hypothesis \ref{H1})  independent of $(B_t)_{t \geq 0}$. $a(\cdot)$ is the positive symmetric operator given 
\footnote{$\sqrt{a(w_s(x))}$ makes sense because as a variance matrix, $a(w_s(x))$ is symmetric non negative definite and its square root is the unique symmetric non negative matrix whose square power identifies with reduction through orthogonal matrix  of the diagonal reduced matrix of $a(w_s(x))$ } by the following Green Kubo formula :  
\begin{equation}\label{asymptvar}
a_{i,j}(x):=\frac 12\sum_{l \in \Z}\int_M\left(F_i(x,\omega)F_j(x,T^{|l|}(\omega))+F_j(x,\omega)F_i(x,T^{|l|}(\omega))\right)d\mu(\omega).
\end{equation}
\end{thm}
Observe that when $(M,T,\mu)$  is invertible (as for the periodic Lorentz gas)
then~\eqref{asymptvar} becomes simply
\begin{equation}\label{asymptvar2}
a_{i,j}(x):=\sum_{l \in \Z}\int_M F_i(x,\omega)F_j(x,T^{l}(\omega))d\mu(\omega).
\end{equation}

\begin{rem}
   \begin{itemize}
       \item The process $(B_{L_t'(0)})_{t\geq 0}$ is a martingale for its natural filtration, this is a consequence of the independence between $(B_t)_{t\geq 0}$ and $(L_s')_{s \geq 0}$. Indeed for any $t\geq 0$, $B_{L_t'(0)}$ is integrable with 
       \begin{align*}
           E(|B_{L'_t(0)}|)=E(E(|B_{L'_t(0)}|\, |(L'_s(0))_{s\geq 0}))=E\left((L'_t(0))^{1/2}\right)E(|B_1|),
       \end{align*}
       and $E\left(B_{L_t'(0)}-B_{L_s'(0)}\, | (B_{L'_s(0)})_{s\leq u}\right)=0$:
       \begin{align*}
           E\left(B_{L_t'(0)}-B_{L_s'(0)}\, | (B_{L'_s(0)})_{s\leq u}\right)&=E\left(E\left(B_{L_t'(0)}-B_{L_s'(0)}\, |(L'_s(0))_{s\geq 0} ,(B_{L'_s(0)})_{s\leq u}\right)\, | (B_{L'_s(0)})_{s\leq u}\right)\\
       \end{align*}
       The latter term $E\left(B_{L_t'(0)}-B_{L_s'(0)}\, |(L'_s(0))_{s\geq 0} ,(B_{L'_s(0)})_{s\leq u}\right)$ is $0$ by independence of $(L_s)_{s\geq 0}$ and $(B_s)_{s\geq 0}$ : to see this it is enough to prove that for any measurable map $f:C^0(\mathbb R,\mathbb R)\times C^0(\mathbb R,\mathbb R^d)$, $E\left(f\left((L'_s(0))_{s\geq 0},(B_u)_{u\leq s}\right)(B_{L'_t(0)-L'_s(0)})_{s\geq 0}\right)=0$.
       This is a consequence of the following :
       \begin{align*}
           &E\left(f\left((L'_s(0))_{s\geq 0},(B_s)_{s\geq 0}\right)(B_{L'_t(0)}-B_{L'_s(0)})_{s\geq 0}\right)\\
           &=E\left(f\left((L'_s(0))_{s\geq 0},(B_u)_{u\leq s}\right)(B_{L'_t(0)}-B_{L'_s(0)})_{s\geq 0}\, | (L'_s(0))_{s\geq 0}\right)\\
           &=E\left(E\left(f\left((L'_s(0))_{s\geq 0},(B_u)_{u\leq s}\right)(B_{L'_t(0)}-B_{L'_s(0)})_{s\geq 0}\, | (L'_s(0))_{s\geq 0}\right)\right)\\
           &=\int E\left( f\left((h_u)_{u\geq 0},(B_{h_u})_{u\leq s}\right)(B_{h_t}-B_{h_s})\right)d\PP_{(L'_s(0))_{s\geq 0}}((h_s)_{s\geq 0})=0,
       \end{align*}
       by independence of $(B_{h_u})_{u\leq s}$ with respect to the random variable $B_{h_t}-B_{h_s}$.
       
       \item When $d=1$, the law of the process $\left(\int_0^t \sqrt{a(w_s(x))}dB_{L_s'(0)}\right)_{t\geq 0}$ can be identified with the law of 
       the process $\left( B_{\int_0^ta(w_s(x))dL_s'(0)}\right)_{t\geq 0}$. To check this, it is enough to verify that both processes have same 
       law conditionally to $(L_s'(0))_{s\geq 0}$. The law of each process conditioned to $(L_s'(0))_{s\geq 0}=(h_s)_{s\geq 0}$ for some 
       continuous map $(h_s)_{s\geq 0}$ satisfies Dubins-Schwarz theorem (see \cite{DuSC65}) with the same quadratic variation ($\langle 
       \int_0^t \sqrt{a(w_s(x))}dB_{h_s} \rangle=\int a(w_s(x))dh_s$) thus they have same law.  
   \end{itemize} 
\end{rem}

\subsection{Slow-fast systems (perturbation by the map)}
We consider the following slow fast system for the transformation : Let  $F: \mathbb R^d\times M\rightarrow \mathbb R^d$ and $\overline F : \mathbb R^d \rightarrow \mathbb R^d$,
 both Lipschitz for their coordinates on $\mathbb R^d$. We say that $(x_t^\epsilon)_{t\in \mathbb{R}_+}$ 
solves the perturbed differential equation if for any $(x,\omega)\in \mathbb R^d\times M$, the map $(x_t^\epsilon(x,\omega))_{t\in \mathbb{R}_+}$ is a solution of the following equation :

\begin{align}\label{eqperturbdifftf}
\left\{
\begin{array}{r l}
&\frac{dx^\epsilon_t}{dt}(x,\omega)=F(x_t^{\epsilon}(x,\omega),T^{\lfloor t/\epsilon\rfloor}\omega)+\Fb(x_t^\epsilon(x,\omega))\\
&x_0^\epsilon(x,\omega)=x 
\end{array}\right.
\end{align}

The next theorem states a limit theorem over the convergence of the solutions $(x_t^\epsilon)_{t\in \mathbb{R}_+}$ of a perturbed differential equation \eqref{eqperturbdifftf} to the averaged map $(w_t)_{t\in \mathbb{R}_+}$ defined as follows : for any $x\in \mathbb R^d$, $(w_t(x))_{t\geq 0}$ is a solution of

\begin{align}\label{eqomega}
\left\{
\begin{array}{r l}
&\frac{dw_t}{dt}(x)=\Fb(w_t(x))\\
&w_0(x)=x.\\
\end{array}\right.
\end{align}

More precisely it provides a limit theorem for the error term $(e^\epsilon_t)_{t\geq 0}$ defined for any $(x,\omega)\in \mathbb R^d\times M$ by
$$
e^\epsilon_t(x,\omega):=x^{\epsilon}_t(x,\omega)-w_t(x).
$$ 

\begin{thm}\label{thmequadif}
Let $\epsilon >0$, and $(M,T,\mu)$ a $\Z$-extension over a probability preserving dynamical system $(\Mp,\tp,\mup)$ satisfying Hypothesis \ref{H2}, with bounded step function $\phi :\Mp \rightarrow \Z$. Let 
$F:\mathbb R^d\times \Mp \times \Z \rightarrow \mathbb R^d$ satisfying the hypotheses of theorem \ref{thmbirkhoff}, such that for any $(\omega,a)\in M$, $F(\cdot,\omega,a)$ is $C^2$ and its derivative $D_1F$ is itself uniformly bounded with uniformly bounded derivative and satisfies : 
$$
\sup_{x \in \mathbb R^d}\sum_{a \in \Z}(1+|a|)^{2(1+\epsilon_0)}\|D_1F(x,.,a)\|_{\infty} <\infty \textit{ and }\sup_{x \in \mathbb R^d}\sum_{a \in \Z}\|D_1F(x,.,a)\|_{lip}<\infty.
$$
Let $\Fb : \mathbb R^d\rightarrow \mathbb R^d$ be a $C^2$ bounded map with bounded derivative. 
Then for any $x \in \mathbb R^d$, the error term $(e_t^\epsilon(x,\cdot))_{t\in [0,T']}$ satisfies a limit theorem for the strong convergence in law  :
\begin{align*}
(\epsilon^{-3/4}e_t^\epsilon(x,\cdot))_{t \in [0,T']}\underset{\epsilon \rightarrow 0}{\overset{\mathcal{L}_\mu,\|.\|_{\infty}}{\rightarrow}}(y_t(x))_{t \in [0,T']},
\end{align*}
where $\mathcal{L}_\mu$ denotes the strong convergence in law for the measure $\mu$ and
where $(y_t(x))_{t\ge 0}$ is the random process given by
\begin{align}\label{eqthmpert}
&(y_t(x))_{t \in [0,T']}:=\left( \int_0^t \sqrt{a(w_s(x))}dB_{L_s'(0)}
\right.\\
&\left. +\int_0^t D\overline F(w_s(x))B_{\int_0^s a(w_u(x))dL_u'(0)}
\exp\left(\int_s^tD\overline F(w_u(x))du\right)ds\right)_{t \in [0,T']},
\end{align}
where $(B_t)_{t\geq 0}$ is a standard Brownian motion (with variance $1$) and $(L_s')_{s \geq 0}$ is the local time associated to a Brownian motion $(B_t')_{t \geq 0}$ with variance $\Sigma$ (where $\Sigma=\sum_{k\in\mathbb Z}E_{\hat{\mup}}( \phi \circ \tp^{|k|} \phi)$ is determined in Hypothesis \ref{H1}), independent of $(B_t)_{t \geq 0}$, and where $a : x \in \mathbb R^d \mapsto a(x)$ is the map defined as
\begin{equation}\label{formulea}
a(x):= \left(\frac 12\sum_{l \in \Z}I_{\mu}\left(F_i(x,T^{|l|}\cdot)F_j(x,\cdot)+F_j(x,T^{|l|}\cdot)F_i(x,\cdot)\right)\right)_{i,j=1,...,d}
\end{equation}
where 
$I_\mu$ denotes the integral with respect to $\mu$, i.e. 
$$
\forall g\in L^1_\mu,\quad I_{\mu}(g):=\int_{M}gd\mu.
$$
\end{thm}
The map $t\mapsto \int_0^t a(w_s(x))dL_s'(0)$ derived from this theorem corresponds to a perturbed variance closed to the one obtained in averaging method for probability invariant dynamical systems. Indeed in the case where $(M,T,\mu)$ is the probability preserving Sinai billiard this variance is given by $\int_0^t a(w_s(x))ds$  (see\cite{phdpene}, \cite{kifer}).
The local time $(L_s'(0))_{s\geq 0}$ is also expected since it is linked to the recurrence of $\overline S_n$ and already appeared naturally in the central limit theorem in 
\cite{PT20} which was proved for maps on $M \simeq \Mp\times \Z$ depending only on the second coordinate, $f(w,l):=\beta(l)$.

\section{Perturbed ergodic sums : proof of the limit theorem.}

In order to ease its reading, we prove the statement of theorem \ref{thmbirkhoff} in dimension $d=1$. The proof in higher dimension follows the same path and we will only mention the alternative quantities on which the techniques apply when necessary. We first prove the following intermediate lemma (stated in dimension $d\geq 1$).

\subsection{Convergence in law.}\label{secbirk}

\begin{prop}\label{propbirkhoff1}
Let $S>0$, on a $\Z$-extension $(M,T,\mu)$ and a map $F : \mathbb R^d \times M \rightarrow \mathbb R^d$ both satisfying the same hypotheses as their counterpart in theorem \ref{thmbirkhoff}, we define the variable $\tilde v_t^\epsilon$ for $t\in [0,S]$, $x\in \mathbb R^d$ and $\omega \in M$ as

$$
\tilde v_t^\epsilon(x,\omega):=\epsilon^{1/4}\sum_{k=1}^{\lfloor t/\epsilon \rfloor} F(w_{\epsilon k}(x),T^{k}\omega).
$$

When $t$ is fixed, this family of random variables satisfies the following convergence in law with respect to the probability measure $\mup$ on $\Mp\simeq M\times\{0\}$ ,
\begin{align*}
\tilde v_t^{\epsilon}(x,\cdot) \underset{\epsilon \rightarrow 0}{\overset{\mathcal{L}_{\mup}}{\rightarrow}} \mathcal N \left(\int_0^t a(w_s(x))dL_s'(0)\right)^{1/2},
\end{align*}
with $a(x)$ given by~\eqref{formulea} and where $\mathcal{N}$ is a standard (variance $1$) centered normal law and $\left((L_s'(u))_{s \geq 0}\right)_{u \in \mathbb{R}}$ is a continuous version of the local time of a Brownian motion $(B'_t)_{t \geq 0}$ of variance $\Sigma$ (with $\Sigma$ as defined in Hypotheses \ref{H1})  independent of $\mathcal{N}$. 
\end{prop}

 In order to prove Proposition~\ref{propbirkhoff1} we introduce some classical results about the regularity of the Transfer operator $P$ of the factor $(\pM,\pt,\pmu)$ (stated in Lemma \ref{anisop}) and mixing properties of the infinite measure preserving dynamical system $(\pM,\pt,\pmu)$  (related to the mixing local limit theorem, see Proposition \ref{spectre})  based on the spectral gap satisfied by the perturbed operator $P_u$ (see Hypothesis \ref{H1}).

\begin{lem}\label{anisop}

Let $\mathcal A_n:=\{g\in \B, \forall x,y \in \hat M, g(x)=g(y) \textit{ whenever }s(x,y)\geq n\}$. 
There is $C>0$ such that, for any bounded map $g\in \A_n$,  
$$
\|P^{2n}(g\cdot)\|_{L(\B,\B)} \leq C\|g\|_{\infty}.
$$
\end{lem}

\begin{proof}
Let $f \in B$, we first determine $\|P^{2n}(gf)\|_{\infty}'$ : 
For some $\Delta_{i,k}$  we denote by $l(\cdot)\in \B$  the map defined by $l(w)=i$ whenever there is $j \in \N$ such that $w \in \Delta_{i,j}$. The positivity of $P$ implies that for $x \in \Delta_{i,k}$, 
\begin{align*}
|P^{2n}(gf)(x)|&\leq |P^{2n}(\|g\|_{\infty}|f|)(x)|\\
&\leq \|g\|_{\infty}\|P^{2n}|f|\|_{\infty}'e^{\epsilon l(x)}.
\end{align*}
Since $P^{2n}$ is uniformly bounded for any $n\in \N$

There is a constant $C>0$ such that
\begin{align}\label{infini}
\|P^{2n}(gf)\|_{\infty}'\leq  C\|g\|_{\infty}\|f\|_\B.
\end{align}

We now give an estimate of  $\|P^{2n}(gf)\|_{lip}'$. Let $i,k\in \N$,
for any $x,y \in \Delta_{i,k}$, we denote by $X$ a bijection between $T^{-2n}(\{x\})$ and $T^{-2n}(\{y\})$ such that for any $z \in T^{-2n}(\{x\})$, $s(z,X(z))=2n+s(x,y)$. 
Once this set, we have for any $f \in \B$ :
\begin{align}
&\left|P^{2n}(gf)(x)-P^{2n}(gf)(y) \right| \nonumber\\
&= \left| \sum_{w \in \pt^{-2n}(\{x\})}e^{-\sum_{i=0}^{2n-1}h(\pt^i w)}(gf)(w)-e^{-\sum_{i=0}^{2n-1}h(\pt^i X(w))}(gf)(X(w)) \right|\nonumber\\
\label{equ2bis}
&\le\sum_{w \in \pt^{-2n}(\{x\})} \left(\mathcal G_w+\mathcal H_w\right),
\end{align}
with
\begin{align*}
\mathcal G_w:=&\left|e^{\sum_{i=0}^{2n-1}h(\pt^i w)}
-e^{\sum_{i=0}^{2n-1}h(\pt^i X(w))}
\right|\|g\|_{\infty}\|f\|_{\infty}'e^{\epsilon l(w)}
\nonumber\\
&\leq \|g\|_{\infty}\|f\|_{\infty}' 
e^{-\sum_{i=0}^{2n-1}h(\pt^i w)}\sum_{i=0}^{2n-1}\|h\|_{lip}\beta^{2n-i+s(x,y)}e^{\epsilon l(w)}\beta^{s(x,y)},
\end{align*}
and
\begin{align*}
\mathcal H_w&:=e^{-\sum_{i=0}^{2n-1}h(\pt^i X(w))}\|g\|_{\infty}|f(w)-f(X(w))|\nonumber\\
&\le e^{-\sum_{i=0}^{2n-1}h(\pt^i X(w))}\|g\|_{\infty}\|f\|_{lip}'e^{\epsilon l(w)}\beta^{s(x,y)+2n}.\nonumber
\end{align*}
Thus, it follows from~\eqref{equ2bis} and from the previous estimates of $\mathcal G_w$ and $\mathcal H_w$ that
\begin{align*}
&\left|P^{2n}(gf)(x)-P^{2n}(gf)(y) \right|\leq \|P^{2n}e^{\epsilon l(\cdot)}\|_{\infty,\Delta_{i,k}}\beta^{s(x,y)}\|g\|_{\infty}\left(\frac{\beta}{1-\beta}\|h\|_{lip} \|f\|_{\infty}'+\|f\|_{lip}'\right).
\end{align*}
Since $\|e^{\epsilon l(\cdot)}\|_\B\leq 1$ and $(P^n)_{n \in \N}$ is a sequence of uniformly bounded operators, There is $C_2>0$ such that
$$
\|P^{2n}(gf)\|_{lip}'\leq C_2\|g\|_{\infty}\|f\|_\B.
$$
This concludes the lemma.

\end{proof}

\begin{prop}\label{spectre}
Let $\e >0$. Let $(\pM,\pt,\pmu)$ be a probabilistic dynamical system with transfer operator $P$ and $\B$ satisfying \ref{H1}.
For $l \in \Z$ and $n\in \N$,

$$
P^n(1_{\{S_n=l\}}\cdot)=\frac{1}{(2\pi \Sigma n)^{1/2}} \Pi_0(\cdot) +G_{n,l}(\cdot)
$$
where $\|G_{n,l}(\cdot)\|_{\B}=O\left( \frac {(1+|l|)^{1+ \e }}{n^{1+\frac \e 2}}\right)$, uniformly in the choice of $l$.
\end{prop}

\begin{proof}

Notice that on $\pM$,
$$
1_{\{S_n=l\}}=\int_{-\pi}^\pi e^{iu(S_n-l)}du.
$$
Fix $b$ such that 
 $\beta_0 >b>0$ where $\beta_0$ is fixed from Hypothesis \ref{H1}, using the identity $P_u^n(\cdot)=P_u^n(e^{iuS_n}\cdot)$ and a change of variable $u=\frac{v}{n^{1/2}}$, 
\begin{align}\label{pert}
P^n(1_{\{S_n=l\}}\cdot)&=\frac{1}{2\pi} \int_{-\pi}^\pi P^n(e^{iu(S_n-l)}\cdot)du \nonumber \\
&=\frac{1}{2\pi} \int_{-b}^b e^{-iul}\lambda_u^n\Pi_u(\cdot)du+O(\alpha^n)\nonumber \\
&=\frac{1}{2\pi n^{1/2}} \int_{-b n^{1/2}}^{b n^{1/2}} e^{-iul/n^{1/2}}\lambda_{u/n^{1/2}}^n\left(\Pi_0(\cdot)+un^{-1/2}\Pi_{0}'(\cdot)\right)du+O(n^{-3/2}),
\end{align}
where the approximations $O(\cdot)$ are uniforms in $l$ for the operator norm $\|.\|_{L(\B,\B)}$.
We then use the following identity \eqref{expapro} from Feller's book \cite{feller}; for any $A,B \in \mathbb{R}$ and $r \in \mathbb{N}$,

\begin{align}\label{expapro}
\left|e^A-\sum_{m=0}^r \frac{B^m}{m!}\right|\leq \left|e^A-e^B\right|+\left|e^B-\frac{B^m}{m!}\right| \leq e^{A \vee B}\left(|A-B|+\left|\frac{B^{r+1}}{(r+1)!}\right|\right). 
\end{align} 
We approximate $\lambda_t$ as some exponential form,
$$
\lambda_{n^{-1/2}t}^n = e^{-\frac{\Sigma} 2 t^2+n\psi(\frac t {n^{1/2}})}
$$
where $\psi(t):=t^2\psi_2(t)+t^{4}\tilde{\psi}(t)$ has a polynomial part $\psi_2(t):= \frac{t}{3!}\psi^{(3)}(0)+\frac{t^2}{4!}\psi^{(4)}(0)$ and an error term $\tilde{\psi}$ continuous with $\tilde{\psi}(0)=0$.\\
Then applying equation \eqref{expapro} with $A:=n\psi(un^{-1/2})$, $B:=u^2\psi_2(un^{-1/2})$ and $r=2$ :
$$
\left|\lambda_{un^{-1/2}}^n-e^{-\frac{\Sigma} 2 u^2}\sum_{m=0}^r\frac{\left(u^2\psi_2(un^{-1/2})\right)^m}{m!}\right|=O\left(e^{-\frac{\Sigma u^2}{4}}\left(\frac{u^4\tilde{\psi}(u n^{-1/2})}{n}+\frac{u^9}{n^{3/2}}\right)\right).
$$
And thus,
$$
\lambda_{un^{-1/2}}^n=e^{-\frac{\Sigma} 2 u^2}\left(1+\frac{u^3}{n^{1/2}}\frac{\psi^{(3)}(0)}{3!}\right)+O\left(e^{-\frac{\Sigma u^2}{4}} Q(u)1/n \right). 
$$
Where $Q$ is polynomial.
We inject this equation into formula \eqref{pert} and get the following expansion for the norm $\|.\|_{L(\B,\B)}$,
\begin{align*}
P^{n}(1_{\{S_n=l\}}\cdot)&=\frac{1}{2\pi n^{1/2}} \int_{-\beta n^{1/2}}^{\beta n^{1/2}} e^{-iul/n^{1/2}}\lambda_{u/n^{1/2}}^n\left(\Pi_0(\cdot)+un^{-1/2}\Pi_{0}'(\cdot)\right)du+O(n^{-3/2})\\
&=\frac{1}{2\pi n^{1/2}} \int_{-\beta n^{1/2}}^{\beta n^{1/2}} \left(1-\frac{iul}{n^{1/2}}+O\left(\left(\frac{ul}{n}\right)^{1+\e}\right)\right)e^{-\frac{\Sigma^2u^2}{2}}\\
&\left(\Pi_0(\cdot)+\frac u {n^{1/2}}\Pi_0'(\cdot)+\frac{u^3}{n^{1/2}}\frac{\psi^{(3)}(0)}{3!}\Pi_0(\cdot)\right)du+O(n^{-3/2})\\
&=\frac{1}{2\pi n^{1/2}} \int_{-\beta n^{1/2}}^{\beta n^{1/2}} \left(1-\frac{iul}{n^{1/2}}\right)\left(\Pi_0(\cdot)+\frac u {n^{1/2}}\Pi_0'(\cdot)+\frac{u^3}{n^{1/2}}\frac{\psi^{(3)}(0)}{3!}\Pi_0(\cdot)\right)\\
&e^{-\frac{\Sigma^2u^2}{2}}du+O\left(\frac{(1+|l|)^{1+\e}}{n^{1+\frac{\e} 2}}\right).\\
\end{align*}
Indeed, $\int_{\mathbb{R}}|u|^m e^{-\frac{(\Sigma u)^2}2}du <\infty$ for all $m \in \N$ and 
$$
\left\{\begin{array}{r l c}
&|e^{-ix}-1-ix|\leq \frac{x^2}{2}\leq \frac{x^{1+\e}}{2}& \textit{ si }x\leq 1\\
&|e^{-ix}-1-ix|\leq (1+|x|)\leq |x|^{1+\e}& \textit{ si }x > 1,
\end{array}\right.
$$

with $O(\cdot)$ uniform in $l\in \mathbb{Z}$.
The terms in $\frac{1}{n^{1/2}}$ with symmetric factors make the integral null. Thus we conclude the lemma,
$$
P^{n}(1_{\{S_n=l\}}\cdot)=\frac{1}{(2\pi \Sigma n)^{1/2}}\Pi_0(\cdot)+G_{n,l}(\cdot)
$$
with
$$
\|G_{n,l}(\cdot)\|_{L(\B,\B)} =O\left( \frac {(1+|l|)^{1+ {\e} }}{n^{1+\frac {\e} 2}}\right).
$$
\end{proof}


\begin{proof}[Proof of Proposition \ref{propbirkhoff1}]
To prove Proposition \ref{propbirkhoff1} we use a criterion of convergence of moments. Inside the proof of \ref{propbirkhoff1} we will state in Lemma \ref{anisop} some approximation of key quantities into product that can be lifted to $(\pM,\pt,\pmu)$. Then the convenient mixing properties on $(\pM,\pt,\pmu)$ will allow us to decorrelate the different products involved and to treat them as i.i.d variable. 
In order to ease further the reading,  we also introduce the following notations all along the proof :
Given $X$ a subset of some vector space, we represent with thick writing $\textbf a$ a vector of at most countable coordinate( $\textbf a \in X^N$ or ${\bf a} \in X^{\N}$), we denote by $\textbf{a}_u^v$ the subvector of coordinates $(a_u,\dots,a_v)\in X^{v-u}$ and by $\underline{\textbf a}_q$ the sum of the coordinates up to $q\in \N$, $\underline{\textbf a}_q=\sum_{k=0}^qa_k$. We also state as convention that $a_{-1}=0$.
As stated before, Proposition \ref{propbirkhoff1} will be proved using a convergence of moment technique and applying Carleman's criterion from \cite[Chapter VII, 3.14]{FE71}.
Thus we investigate here an expression of the moment\footnote{when the dimension of the space is $d>1$ we look at the moment of $\langle u,\tilde v_t^\epsilon(x,\cdot)\rangle$ for any $u \in \mathbb R^d$ instead, we then replace the quantity $\f(x,\cdot)$ appearing in the moment by $\langle u,\f(x,\cdot)\rangle$ through the rest of the proof} of $N$-th order of $\tilde v_t^\epsilon(x,\cdot)$  for some fixed $x \in \mathbb{R}$. 
\begin{align*}
&E_{\mup}\left( \tilde v_t^\epsilon(x,\cdot)^N  \right) =E_{\mup}\left( \left(\epsilon^{1/4}\sum_{k=0}^{\lfloor \frac t \epsilon \rfloor}\sum_{p \in \Z}\F(w_{\epsilon k}(x),\tp^k(w),p)1_{\{S_k=p\}} \right)^N \right)\\
&=\epsilon^{N/4}\sum_{ 0 \leq k_1, \dots, k_N \leq \lfloor t/\epsilon \rfloor}\sum_{\textbf{p} \in \Z^N}E_{\mup}\left( \prod_{i=1}^N \left( \F(w_{\epsilon k_i}(x),\tp^{k_i}(\cdot),p_i)1_{\{S_{k_i}=p_i\}} \right) \right)\\
&=\epsilon^{N/4}\sum_{q=1}^N \sum_{\textbf{N}\in (\N^*)^q,\underline{\textbf{N}}_q=N}C_{\Ng}A_{\epsilon,q,\Ng}(\F),
\end{align*}
In the product we gather the indexes $k_i$ which are identical and we set $C_{\Ng}$ the number of vectors $\textbf{k}$ having $N_i$ coordinates of value $i$ for $i\in [|1,q|]$ and $\textbf{N}\in \{\textbf{N}\in \mathbb N^q, \sum_{i=1}^q N_i=N\}$: 
$$
C_{\Ng}:=|\{\phi \in [|1,q|]^{[|1,N|]}, |\phi^{-1}(\{i\})|=N_i \}|.
$$
As for the quantity $A_{\epsilon,q,\Ng}(\F)$, it is defined as
\begin{align*}
A_{\epsilon,q,\Ng}(\F)&:= \sum_{1 \leq n_1 <\dots< n_q \leq \lfloor t/\epsilon \rfloor}\sum_{\textbf a \in \mathbb{Z}^q}E_{\mup}\left( \prod_{i=1}^q \F(w_{\epsilon n_i}(x),\tp^{n_i}(\cdot),a_i)^{N_i}1_{\{\overline S_{n_i}=a_i\}} \right)\\
&=\sum_{\textbf{l}\in (\N^*)^q, \underline{\textbf{l}}_q\leq\lfloor t/\epsilon \rfloor}\sum_{\textbf a \in \mathbb{Z}^q}\Xi_{\epsilon,q,l_i,\Ng}(\F,w_{\epsilon^{-1} \ls_1}(x),\dots,w_{\epsilon^{-1} \ls_q}(x),\textbf a)
\end{align*}
with for $\F \in \B$, $\textbf{z}\in \mathbb R^q$ and $\textbf{a} \in \mathbb{Z}^q$, 
$$
\Xi_{\epsilon,q,\textbf{l},i,\Ng}(\F,\textbf{z},\textbf{a}):=E_{\mup}\left( \left( \prod_{j=1}^q \F(z_i,\tp^{l_i}(\cdot),a_i)^{N_i}1_{\{\overline S_{l_i}=a_i-a_{i-1}\}} \right)\circ \tp^{\ls_{i-1}} \right).
$$
In what follows, we introduce the quantity 
$$
k_\epsilon:= \lceil ( \log(\frac S \epsilon))^2 \rceil
$$
this quantity will allow us to approximate with arbitrary precision when $\epsilon \to 0$
the terms $A_{\epsilon,q,\Ng}(\F)$ in the following Lemma \ref{aproxpro} in order to place the quantities on $(\pM,\pt,\pmu)$ and use the decorrelation properties.
\begin{lem}\label{aproxpro}
For any $x\in \mathbb{R}^d$, $a \in \Z$, let $(g_{k_\epsilon}(x,.,a))_{\epsilon}$ the map defined by $g_n$ in equation \eqref{unif} as an approximation of $g:=f(x,.,a)$ and satisfying Hypothesis \ref{H2}. Then 
\begin{align}\label{eqmajo}
\sup_{\textbf{z}\in (\mathbb R^d)^q} \sum_{\textbf{a} \in \mathbb{Z}^q}\left|\Xi_{\epsilon,q,\ls,i,\Ng}(\F,\textbf{z},\textbf{a})-\hat \Xi_{\epsilon,q,\ls,i,\Ng}(g_{k_\epsilon},\textbf{z},\textbf{a})\right|\underset{\epsilon \rightarrow 0}{=}O\left(q\beta_0^{(\ln(1/\epsilon))^2} \left(\sup_{x\in \mathbb R}\sum_{a \in \Z}(\|f(x,.,a)\|_{lip}+\|f(x,.,a)\|_{\infty})\right)^N\right).
\end{align}
with $\hat \Xi_{\epsilon,q,\ls,i,\Ng}(g_{k_\epsilon},\textbf{z},\textbf{a}):=E_{\pmu}\left(\left( \prod_{i=1}^q g_{k_\epsilon}(z_i,\pt^{l_i}(\cdot),a_i)^{N_i}1_{\{S_{l_i}=a_i-a_{i-1}\}}\circ \pt^{k_\epsilon} \right)\circ \pt^{\ls_{i-1}} \right)$.
Thus, when $k_\epsilon= \lceil ( \log(\frac S \epsilon))^2 \rceil$,

\begin{align}\label{majo}
\lim_{\epsilon \rightarrow 0}\left|A_{\epsilon,q,\Ng}(f)-\hat A_{\epsilon,q,\Ng}(g_{k_\epsilon})\right|=0,
\end{align}
 where $\hat A_{\epsilon,q,\Ng}(g_{k_\epsilon}):=\sum_{\textbf{l}\in (\N^*)^q, \underline{\textbf{l}}_q\leq\lfloor t/\epsilon \rfloor}\sum_{\textbf{a} \in \mathbb{Z}^q}E_{\pmu}\left(\left( \prod_{i=1}^q g_{k_\epsilon}(x_i,\pt^{l_i}(\cdot),a_i)^{N_i}1_{\{S_{l_i}=a_i-a_{i-1}\}}\circ \pt^{k_\epsilon} \right)\circ \pt^{\ls_{i-1}} \right)$.
  
\end{lem}
\begin{proof}
Thanks to the invariance of $\mup$ by $\tp$, we can compose by $\tp^{k_\epsilon}$ in the expectancy of $A_{\epsilon,q,\Ng}(\F)$ and lift the quantity inside on $(M_1,T_1,\mu_1)$ using the projections $\pi_1$ and $\pi_2$. Thus applying the formula \eqref{unif}, we get for any $1 \leq i_0 \leq q$,

\begin{align*}
 &\sum_{\textbf{a} \in \mathbb{Z}^q}\left\|\right.E_{\mu_1}\left(\left.\left( \prod_{i=1}^{i_0-1} \F(x_i,\pi_1 \circ T_1^{l_i +k_\epsilon}(\cdot),a_i)^{N_i}1_{\{\overline S_{l_i}=a_i-a_{i-1}\}}\circ \pi_2 \circ T_1^{k_\epsilon} \right)\circ T_1^{\ls_{i-1}} \right.\right. \\ 
&\left( \F(x_{i_0},T_1^{l_{i_0+k_\epsilon}}(\cdot),p_{i_0})^{N_{i_0}}-g_{k_\epsilon}(x_{i_0}, \pi_2 \circ T_1^{l_{i_0}}(\cdot),a_{i_0})^{N_{i_0}})1_{\{S_{l_{i_0}}=a_{i_0}-a_{i_0-1}\}}\circ \pi_2\circ T_1^{k_\epsilon} \right)\circ T_1^{\ls_{i_0-1}} \\
&\left. \left. \left( \prod_{i=i_0+1}^q g_{k_\epsilon}(x_i, \pi_2\circ T_1^{l_i}(\cdot),a_i)^{N_i}1_{\{S_{l_i}=a_i-a_{i-1}\}}\circ \pi_2 \circ T_1^{k_\epsilon} \right)\circ \tp^{\ls_{i-1}} 
 \right)\right\|_{\infty} \\
& \leq \beta_0^{k_\epsilon}\sum_{a \in \Z}\|\F(x_{i_0},.,a)\|_{lip} \sum_{\textbf{l}\in (\N^*)^q, \underline{\textbf{l}}_q\leq\lfloor t/\epsilon \rfloor}\sum_{\textbf{a} \in \mathbb{Z}^q}\prod_{i\neq i_0}\|f(x_i,.,a_i)\|_\infty^{N_i}\\
&  =O\left( \beta_0^{k_\epsilon} \prod_{i=1}^N\left(\sum_{a \in \Z}(\|f(x_i,.,a)\|_{lip}+\|f(x_i,.,a)\|_{\infty})\right)\right).
\end{align*}

Then, by chain rule, we get the upper bound \eqref{eqmajo} we looked for and we deduce \eqref{majo}.
\end{proof}

Thanks to Lemma \ref{aproxpro}, it remains to estimate the following quantity
\begin{align*}
    \hat A_{\epsilon,q,\Ng}(g_{k_\epsilon})=\sum_{1 \leq l_1,\dots,l_q \leq n_{\epsilon}, \ls_q \leq n_{\epsilon}}b_{a_0,m,\epsilon,\textbf{l},\Ng}(g_{k_\epsilon})
\end{align*}

where the quantity $b_{a_0,m,\epsilon,\textbf{l},\Ng}(f)$ is defined for any $f\in \B$ and $\textbf{l}\in \N^q$ by
\begin{align*}
&b_{a_0,m,\epsilon,\textbf{l},\Ng}(f):=\\
&\sum_{a_1,\dots,a_q \in \mathbb Z}E_{\pmu}\left(f\left( \prod_{i=1}^q g_{k_\epsilon}(w_{\epsilon^{-1} (\ls_i+m)}(x),\pt^{l_i}(\cdot),a_i)^{N_i}1_{\{S_{l_i}=a_i-a_{i-1}\}}\circ \pt^{k_\epsilon} \right)\circ \pt^{\ls_{i-1}} \right).
\end{align*}

We use some classical trick : since $\hat \mup$ is $\pt$-invariant, then $E_{\pmu}(P(\cdot))=E_{\pmu}(\cdot)$, and we can apply $P^{\ls_q+k_\epsilon}$ in the expectancy above and iterate the identity $P^m(G\circ \pt^m f)=G P^m(f)$ to obtain the following,

\begin{align}\label{expdecor}
b_{a_0,m,\epsilon,\textbf{l},\Ng}(f)&:= E_{\pmu}\left(P^{k_\epsilon}\left(  (g_{k_\epsilon}(w_{\epsilon^{-1} (\ls_q+m)}(x),.,a_q)^{N_q}H_{k_\epsilon,l_q,a_q-a_{q-1}})\circ
\dots \right.\right. \nonumber\\
& \left. \left.\circ (g_{k_\epsilon}(w_{\epsilon^{-1} (\ls_1+m)}(x),.,a_1)^{N_1}H_{k_\epsilon,l_1,a_1-a_0}) (f) \right) \right),
\end{align}
with
$$
H_{k_\epsilon,m,a}(\cdot):=P^m(1_{\{S_m=a\}}\circ \pt^{k_\epsilon}\cdot)
$$
for any $m \geq 2k_\epsilon$. This latter term may be written as
\begin{align*}
H_{k_\epsilon,m,a}(\cdot)&=\sum_{l \in \mathbb{Z}}1_{\{S_{k_\epsilon}=l\}}P^{m-k_\epsilon}(1_{\{S_{m-k_\epsilon}=a-l\}}P^{k_\epsilon}(\cdot))\\
&=\sum_{l \in \mathbb{Z}}\sum_{l' \in \mathbb{Z}}1_{\{S_{k_\epsilon}=l\}}P^{m-2k_\epsilon}\left(1_{\{S_{m-2k_\epsilon}=a-l-l'\}}P^{k_\epsilon}(1_{\{S_{k_\epsilon}=l'\}}P^{k_\epsilon}(\cdot))\right).
\end{align*}
As in \cite{PT20}, we split $H_{k_\epsilon,m,a}(\cdot)$ as follows in order to apply the mixing property \ref{spectre} :
For any $m \geq 2k_\epsilon$, we denote by $Q_m(\cdot)$ the operator defined as 
$$
Q_{m,a}(\cdot):=P^m(1_{\{S_m=a\}}\cdot).
$$

This operator can be split according to Proposition \ref{spectre} into 
$$
Q_{m,a}(\cdot):=Q_{m,a}^{(0)}(\cdot)+Q_{m,a}^{(1)}(\cdot)\textit{ where }Q_{m,a}^{(0)}(\cdot)=\frac{1}{(2\pi \Sigma m)^{1/2}} \Pi_0(\cdot)  \textit{ and } Q_{m,a}^{(1)}(\cdot)=G_{m,a}(\cdot).
$$
We also introduce the following operator
$$
H_{k_\epsilon,m,a}^{(1)}(\cdot):=
\left\{ \begin{array}{r l c}
&H_{k_\epsilon,m,a} \textit{ if }m \leq 2k_\epsilon\\
&\sum_{l \in \mathbb{Z}}1_{\{S_{k_\epsilon}=l\}}\sum_{l' \in \Z}G_{m-2k_\epsilon,a-l}(P^{k_\epsilon}(1_{\{S_{k_\epsilon}=l'\}}(P^{k_\epsilon}\cdot)) \textit{ otherwise}
\end{array}\right.
$$ 
and when $2k_\epsilon \leq m$,
$$
H_{k_\epsilon,m,a}^{(0)}(\cdot):=\sum_{l \in \mathbb{Z}}1_{\{S_{k_\epsilon}=l\}}Q_{m-k_\epsilon,a-l}^{(0)}(P^{k_\epsilon}(\cdot))=Q_{m-k_\epsilon,0}^{(0)}.
$$ 

The last notations we introduce are the two following sets
 $E(\textbf{l}):=\{(e_i)_{1\leq i \leq q}, \textit{ if } e_i=0 \textit{ then } l_i \geq 2k_\epsilon\}$ and\\
 $\mathcal G(\textbf{e}):=\{\textbf{l}, \sum_i l_i \leq \lfloor t/\epsilon \rfloor, \textit{ and } \textbf{e} \in E(\textbf l)\}$,
and we can now decompose further $\hat A_{\epsilon,q,\Ng}(g_{k_\epsilon})$ as follows,

\begin{align*}
\hat A_{\epsilon,q,\Ng}(g_{k_\epsilon})=\sum_{1 \leq l_1,\dots,l_q \leq n_{\epsilon}, \ls_q \leq n_{\epsilon}}b_{a_0,m,\epsilon,\textbf{l},\Ng}(g_{k_\epsilon})=\sum_{\textbf{e} \in \{0,1\}^q}\sum_{\textbf{l}\in \mathcal{G}(\textbf{e})} b_{a_0,m,\epsilon,\textbf{l},\Ng}^{\textbf{e}}(g_{k_\epsilon}),
\end{align*}
where for each $\textbf{e}=(e_1,\dots, e_q)$ and $f\in \B$,
\begin{align}
&b_{a_0,m,\epsilon,\textbf{l},\Ng}^{\textbf{e}}(f):=\sum_{a_1,\dots,a_q \in \Z}E_{\pmu}\left(P^{2k_\epsilon}\left(  (g_{k_\epsilon}(w_{\epsilon (\ls_q+m)}(x),.,a_q)^{N_q}H_{k_\epsilon,l_q,a_q-a_{q-1}}^{(e_q)})\circ
\dots \right.\right.\nonumber\\
& \left. \left.\circ (g_{k_\epsilon}(w_{\epsilon (\ls_1+m)}(x),.,a_1)^{N_1}H_{k_\epsilon,l_1,a_1-a_0}^{(e_1)}) (f) \right) \right).\label{depar}
\end{align}

In what follows, we look for the couples $(\Ng,\textbf e)$ that will provide some weight in the computation of the $Nth$ moment and get rid of the others. 
The term $Q_{m,a}^{(0)}(\cdot)=\frac{1}{(2\pi \Sigma n)^{1/2}}E_{\pmu}(\cdot)1$ appearing in $H_{k_\epsilon,m,a}^{(0)}(\cdot)$ will enable us to apply decorrelation. Notice from the expression in \eqref{depar} that for $0 \leq i \leq q$ :
\begin{align}\label{magik}
&b_{a_0,m,\epsilon,(\textbf{l},l_i,\textbf{l}'),\Ng}^{(\textbf{e},0,\textbf{e'})}(G)=b_{a_0,m,\epsilon,\textbf{l},\Ng}^{\textbf{e}}(G)b_{0,m+\ls_{i-1},\epsilon,(l_i,\textbf{l}'),(N_i,\Ng')}^{(0,\textbf{e}')}(1)\nonumber \\
&=\frac{1}{(2\pi \Sigma l_i)^{1/2}}b_{a_0,m,\epsilon,\textbf{l},\Ng}^{\textbf{e}}(G)\sum_{a_i \in \mathbb{Z}}b_{a_i,m+\ls_i,\epsilon,\textbf{l}',\Ng'}^{\textbf{e}'}\left(g_{k_\epsilon}(w_{\epsilon (m+\ls_i)}(x),.,a_i)^{N_i}\right).
\end{align}
Let $\textbf{e}\in \{0,1\}^q$, we denote by $\alpha_1<\dots<\alpha_K$ the indexes such that $e_{\alpha_i}=0$, we fix the following conventions $\alpha_{K+1}=q+1$ and $e_{q+1}=0$. Now we fix $\textbf{l} \in \mathcal G(\textbf{e})$ and we decompose $\frac 1 {n^{N/4}}b_{0,0,\epsilon,\textbf{l},\Ng}^{\textbf{e}}(1)$ in
\begin{align}\label{unpourotus}
\epsilon^{N/4}b_{a_0,0,\epsilon,\textbf{l},\Ng}^{\textbf{e}}(1)&=\frac{1}{(2\pi \Sigma)^{K/2}}\epsilon^{\frac{N_1+\dots+N_{\alpha_1-1}}{4}}b_{a_0,0,\epsilon, \textbf{l}_1^{\alpha_1-1},\textbf{N}_1^{\alpha_1-1}}^{(1,\dots,1)}(1)\nonumber\\
&\times\prod_{i=1}^K\frac{1}{l_{\alpha_i}^{1/2}}\epsilon^{\frac{\underline{\textbf{N}}_{\alpha_i}^{\alpha_{i+1}-1}}{4}}\sum_{a\in \Z} b_{a,\ls_{\alpha_i},\epsilon,\textbf{l}_{\alpha_i+1}^{\alpha_{i+1}-1},\textbf{N}_{\alpha_i+1}^{\alpha_{i+1}-1}}^{(1,\dots,1)}\left(g_{k_\epsilon}(w_{\epsilon \ls_{\alpha_i}(x)},.,a)^{N_{\alpha_i}}\right).
\end{align}
In what follows, we prove that the mass $\sum_{a\in \Z} b_{a,\ls_{\alpha_i},\epsilon,\textbf{l}_{\alpha_i+1}^{\alpha_{i+1}-1},\textbf{N}_{\alpha_i+1}^{\alpha_{i+1}-1}}^{(1,\dots,1)}(\cdot)$ is negligible.
Let $K_1,\dots,K_r$ be the indexes in $[|\alpha_i+1,\dots,\alpha_{i+1}-1|]$ such that $l_{K_i}\geq 2k_\epsilon$ (with the convention $K_0:=\alpha_i+1$ and $K_{r+1}:=\alpha_{i+1}$) , and 
\begin{align*}
B_{i,j,m}(\textbf{a}_{i-1}^j) :&=\left(g_{k_\epsilon}(w_{\epsilon (\sum_{k=i}^jl_k+m)}(x),.,a_j)^{N_j}H_{k_\epsilon,l_j,a_j-a_{j-1}}\right) \circ \dots\\
&   \circ \left(g_{k_\epsilon}(w_{\epsilon (l_i+m)}(x),.,a_i)^{N_i} H_{k_\epsilon,l_{i},a_{i}-a_{i-1}}(\cdot)\right). 
\end{align*}

For $h \in \B$,
\begin{align}
&b_{a,m+\ls_{\alpha_i},\epsilon,\textbf{l}_{\alpha_i+1}^{\alpha_{i+1}-1},\textbf{N}_{\alpha_i+1}^{\alpha_{i+1}-1}}^{(1,\dots,1)}(h)=\sum_{\textbf{a}_{\alpha_i+1}^{\alpha_{i+1}-1} \in \Z^{\alpha_{i+1}-1-\alpha_{i}-1}}E_{\pmu}\left(P^{2k_\epsilon}\left( \left[B_{K_r+1,K_{r+1}-1,\ls_{K_r}}(\textbf{a}_{K_r}^{K_{r+1}-1})\circ \right. \right. \right. \nonumber\\
& \left. \left( \sum_{l \in \mathbb{Z}} \left(1_{\{S_{k_\epsilon}=l\}}g_{k_\epsilon}(w_{\epsilon (\ls_{K_r}+m)}(x),.,a_{K_r})^{N_{K_r}}\right.\right.\right.\left.\left.\left. \sum_{l'\in \Z}Q_{l_{K_r}-2k_\epsilon,a_{K_r}-a_{K_r-1}-l-l'}^{(1)} \circ P^{k_\epsilon}(1_{\{S_{k_\epsilon}=l'\}}P^{k_\epsilon}(\cdot)))\right)\right) \right]\circ 
\dots  \nonumber \\
&\circ \left[B_{K_1+1,K_{2}-1,\ls_{K_1}}(\textbf{a}_{K_1}^{K_{2}-1})
\left( 
\sum_{l \in \mathbb{Z}} \left(1_{\{S_{k_\epsilon}=l\}}g_{k_\epsilon}(w_{\epsilon (\ls_{K_1}+m)}(x),.,a_{K_1})^{N_{K_1}}\right. \right.\right. \nonumber\\
&\left. \left. \left.\left.\left.\sum_{l'\in \Z}Q_{l_{K_1}-k_\epsilon,a_{K_1}-a_{K_1-1}-l-l'}^{(1)}\left(P^{k_\epsilon}(1_{\{S_{k_\epsilon}=l'\}}P^{k_\epsilon}(\cdot))\right)\right) \right) \right]\circ   B_{K_0,K_1-1,m+\ls_{\alpha_i}}(a,\textbf{a}_{K_0}^{K_{1}-1})(h)   \right)\right)\nonumber
\end{align}
Applying the finite horizon hypothesis in the sense that there is $D>0$ such that $\|\phi \|_{\infty} \leq D$, and we deduced that $1_{\{S_{k_\epsilon}=l\}}=0$ for any $|l| >Dk_\epsilon$, we can rewrite the above quantity as follows
\begin{align}
&\sum_{\textbf{a}_{\alpha_i+1}^{\alpha_{i+1}-1} \in \Z^{\alpha_{i+1}-1-\alpha_{i}-1}} \sum_{l_{K_1}'\dots, l_{K_{r+1}}'\in \Z}E_{\pmu}\left( \left[Q_{k_\epsilon,l_{K_{r+1}}'}\circ P^{k_\epsilon}\circ B_{K_r+1,K_{r+1}-1,\ls_{K_r}}(\textbf{a}_{K_r}^{K_{r+1}-1})\circ \right. \right. \nonumber \\
&\left. \left( \sum_{ |l| \leq Dk_\epsilon } 1_{\{S_{k_\epsilon=l}\}}g_{k_\epsilon}(w_{\epsilon (\ls_{K_r}+m)}(x),.,a_{K_r})^{N_{K_r}}Q_{l_{K_r}-2k_\epsilon,a_{K_r}-a_{K_r-1}-l-l_{K_r}'}^{(1)}\right) \right]\circ \dots \nonumber\\
&\circ \left[  Q_{k_\epsilon,l_{K_2}'}\circ  P^{k_\epsilon}\circ B_{K_1+1,K_{2}-1,\ls_{K_1}}(\textbf{a}_{K_1}^{K_{2}-1}) \circ \right.\nonumber \\
&\left.\left( \sum_{ |l| \leq Dk_\epsilon} 1_{\{S_{k_\epsilon=l}\}}g_{k_\epsilon}(w_{\epsilon (\ls_{K_1}+m)}(x),.,a_{K_1})^{N_{K_1}}Q_{l_{K_1}-2k_\epsilon,a_{K_1}-a_{K_1-1}-l-l_{K_1}'}^{(1)}\right)\right]\nonumber\\
&\left. \circ Q_{k_\epsilon,l_{K_1}'}\circ  P^{k_\epsilon}\circ B_{K_0,K_1-1,m+\ls_{\alpha_i}}(a,\textbf{a}_{K_0}^{K_{1}-1})(h)\right),\label{groseq}
\end{align}
As in \eqref{expdecor}, we can use the identity $P^r(G\circ \pt^r f)=G P^r(f)$  to express\\
$Q_{k_\epsilon,l_{K_{i+1}}'}\circ P^{k_\epsilon}\circ B_{K_i+1,K_{i+1}-1,\ls_{K_i}}(\textbf{a}_{K_i}^{K_{i+1}-1})( 1_{\{S_{k_\epsilon}=l\}}g_{k_\epsilon}(w_{\epsilon (\ls_{K_i}+m)}(x),.,a_{K_i})^{N_{K_i}}\cdot)$ as
\begin{align}\label{borninf}
Q_{k_\epsilon,l_{K_{i+1}}'}\circ P^{k_\epsilon}\circ& B_{K_i+1,K_{i+1}-1,\ls_{K_i}}(\textbf{a}_{K_i}^{K_{i+1}-1})( 1_{\{S_{k_\epsilon}=l\}}g_{k_\epsilon}(w_{\epsilon (\ls_{K_i}+m)}(x),.,a_{K_i})^{N_{K_i}}\cdot):=\\
&P^{2k_\epsilon+\sum_{u=K_i+1}^{K_{i+1}}l_u}\left(h_{\epsilon,m}(\textbf{a}_{K_i}^{K_{i+1}-1}).\right),
\end{align}
with $h_{\epsilon,m}(\textbf{a}_{K_i}^{K_{i+1}-1})$ defined as follows
\begin{align*}
h_{\epsilon,m}(\textbf{a}_{K_i}^{K_{i+1}-1})&:=1_{\{S_{k_\epsilon=l_{K_{i+1}}'}\}}\circ \hat T^{k_\epsilon+\ls_{K_i+1}^{K_{i+1}}}\\
&\prod_{j=K_i}^{K_{i+1}-1}\left( g_{k_\epsilon}(w_{\epsilon^{-1} (\ls_j+m)}(x),\pt^{l_j-l_{j-1}}(\cdot),a_j)^{N_j}1_{\{S_{l_j}=a_j-a_{j-1}\}}\circ \hat T^{k_\epsilon}\right)^{N_j}\circ \hat T^{\ls_{K_i+1}^{j-1}}\\
&1_{\{S_{k_\epsilon}=l\}}g_{k_\epsilon}(w_{\epsilon^{-1} (\ls_{K_i}+m)}(x),.,a_{K_i})^{N_{K_i}}.
\end{align*}

The map $h_{\epsilon,m}(\textbf{a}_{K_i}^{K_{i+1}})$ is an element of $\B$ such that for any $x,y \in \pM$, 
\begin{align*}
s(x,y)\geq 2k_\epsilon+\sum_{u=K_i+1}^{K_{i+1}}l_u \Rightarrow h_{\epsilon,m}(\textbf{a}_{K_i}^{K_{i+1}})(x)=h_{\epsilon,m}(\textbf{a}_{K_i}^{K_{i+1}})(y)
\end{align*}
and for any $\epsilon,m$ and $a_{K_i}\in \Z$,
\begin{align}\label{hntruc}
\|h_{\epsilon,m}(\textbf{a}_{K_i}^{K_{i+1}})\|_{\infty}\leq \prod_{j=K_i}^{K_{i+1}}\left\|g_{k_\epsilon}(w_{\epsilon (\ls_{j}+m)}(x),.,a_{j})\right\|_{\infty}^{N_{j}}.
\end{align}
We apply Lemma \ref{anisop} on equation \eqref{borninf}, and using Proposition \ref{spectre}, along with equation \eqref{hntruc}  we deduce the following estimate of the terms in formula \eqref{groseq}:
\begin{align}
&\sum_{l \in \Z}\sum_{l_{K_{i+1}}'\in \Z}\left\| Q_{k_\epsilon,l_{K_{i+1}-1}'}\circ P^{k_\epsilon}\circ B_{K_i+1,K_{i+1}-1,\ls_{K_i}}(\textbf{a}_{K_i}^{K_{i+1}-1})\circ \right. \nonumber\\
& \left. \left( \sum_{l \in \mathbb{Z}} 1_{\{S_{k_\epsilon}=l\}}g_{k_\epsilon}(w_{\epsilon(\ls_{K_i}+m)}(x),.,a_{K_i})^{N_{K_i}}Q_{l_{K_i}-2k_\epsilon,a_{K_i}-a_{K_i-1}-l-l_{K_i}'}^{(1)}\right)\right\|_{L(\B,\B)} \label{groseqbis}\\
&\leq \sum_{|l| \leq Dk_\epsilon}\sum_{|l_{K_{i+1}}'|\leq Dk_\epsilon}  \prod_{j=K_i}^{K_{i+1}}\left\|g_{k_\epsilon}(w_{\epsilon (\ls_{j}+m)}(x),.,a_{j})\right\|_{\infty}^{N_{j}}O\left(\frac{|a_{K_i}-a_{K_i-1}-l-l_{K_i}'|^{1+\e}}{|l_{K_i}-2k_\epsilon|^{1+\frac {\e} 2}}\right)\, .\nonumber
\end{align}
It follows that
\begin{align*}
\left|\eqref{groseqbis}\right|
&\leq (2Dk_\epsilon)^{2} \prod_{j=K_i}^{K_{i+1}}\left\|g_{k_\epsilon}(w_{\epsilon (\ls_{j}+m)}(x),.,a_{j})\right\|_{\infty}^{N_{j}} O\left(\frac{(||a_{K_i}|+1|||a_{K_i-1}|+1||1+k_\epsilon|^2)^{1+\e}}{|l_{K_i}-2k_\epsilon|^{1+\frac {\e} 2}}\right).
\end{align*}

We now estimate equation \eqref{groseq} using the above majoration as well and the convergence (by assumption) of the sum  $\sum_{a \in \Z}(1+|a|)^{2(1+\e)}\|g_{k_\epsilon}\|_{\infty}$:

\begin{align}
&\sum_{a \in \Z}\sum_{\textbf{l}\in \mathcal{G}_{\alpha_i+1,\alpha_{i+1}-1}(\textbf{e}) }\| b_{a,m+\ls_{\alpha_i},\epsilon,\textbf{l}_{\alpha_i+1}^{\alpha_{i+1}-1},\textbf{N}_{\alpha_i+1}^{\alpha_{i+1}-1}}^{(1,\dots,1)}(\sum_{l \in \Z}Q_{k_\epsilon,l}(g_{k_\epsilon}(w_{\epsilon \ls_{\alpha_i}(x)},.,a)^{N_{\alpha_i}}))\|_{B}\nonumber \\
&\leq k_\epsilon^q(2D(1+k_\epsilon))^{r(2+1+\e)} \sup_{x\in \mathbb{R}}\sum_{\textbf{a}_{\alpha_i+1}^{\alpha_{i+1}-1} \in \Z^{\alpha_{i+1}-2-\alpha_i}}\prod_{u=\alpha_i+1}^{\alpha_{i+1}-1}(1+|a_u|)^{2(1+\e)}\|g_{k_\epsilon}(x,.,a_u)\|_{\infty}^{N_u}\nonumber\\
&\sum_{k_\epsilon \leq l_{K_1}\dots, l_{K_{r+1}}} O(\prod_{i=K_1}^{K_r} (l_i-2k_\epsilon)^{-(1+\frac{\e}{2})})\nonumber\\
&=O(k_\epsilon^{q+r(2+1+\e)}).\label{erreurptk}
\end{align}
Where $\mathcal G_{i,j}(\textbf{e}):=\{(\textbf{l}_i^j), \textbf{l} \in \mathcal G (\textbf{e})\}$.

Thus for any $1 \leq i \leq K$, 
\begin{align*}
&\epsilon^\frac{\underline{\textbf{N}}_{\alpha_i}^{\alpha_{i+1}-1}}{4}\sum_{\textbf l \in \mathcal{G}_{\alpha_i+1,\alpha_{i+1}-1}(\textbf{e})}\left\|\sum_{a\in \Z} b_{a,m+\ls_{\alpha_i},\epsilon,\textbf{l}_{\alpha_i+1}^{\alpha_{i+1}-1},\textbf{N}_{\alpha_i+1}^{\alpha_{i+1}-1}}^{(1,\dots,1)}(\sum_{l \in \Z}Q_{k_\epsilon,l}(g_{k_\epsilon}(w_{\epsilon \ls_{\alpha_i}(x)},.,a)^{N_{\alpha_i}}))\right\|_{B}\\
&=O(k_\epsilon^{q+r(2+1+\e)}\epsilon^{\frac{\underline{\textbf{N}}_{\alpha_i}^{\alpha_{i+1}-1}}{4}-\frac 1 2}).
\end{align*}
The latter converges to $0$ with rate in $O(\epsilon^{1/8})$ whenever $\underline{\textbf{N}}_{\alpha_i}^{\alpha_{i+1}-1}\geq 3$ and is in $O(k_\epsilon^{q+r(2+1+\e)})$ when $\underline{\textbf{N}}_{\alpha_i}^{\alpha_{i+1}-1}=2$.\\
finally, in the special case where $\underline{\textbf{N}}_{\alpha_i}^{\alpha_{i+1}-1}=1$ (equivalently $\alpha_{i+1}=1+\alpha_i$ and $N_{\alpha_i}=1$), we use the fact that $g_{k_\epsilon}$ has by assumption zero mean. 
$$
b_{a,m,\epsilon,(l_{\alpha_i},l',\textbf{l}),(1,N_2,\textbf{N})}^{(0,0,\textbf{e})}(G)=E_{\pmu}\left(g_{k_\epsilon}(w_{\epsilon (l_{\alpha_i}+m)}(x),.,a_q)E_{\pmu}(\dots)\right)=0.
$$ 

Thus, summing formula \eqref{unpourotus} over all possible elements $\textbf{l}$, 
\begin{align}\label{obscur}
\frac{1}{\epsilon^{q/4}}\sum_{\textbf{l}_1^q,\ls_q\leq \lfloor t/\epsilon\rfloor}|b_{a_0,0,\epsilon,\textbf{l},\Ng}^{\textbf{e}}(1)| \underset{\epsilon \rightarrow 0}{\overset {L^1_{\pmu}} {\rightarrow}} 0,
\end{align}
as soon as there is an index $1 \leq i \leq K$ such that $\underline{\textbf{N}}_{\alpha_i}^{\alpha_{i+1}-1}\geq 3$.

The couples $(\textbf{e},\textbf{N})$ such that $\epsilon^{q/4}\sum_{l_1,\dots,l_q,\ls_q\leq \lfloor t /\epsilon\rfloor}b_{a_0,0,\epsilon,\textbf{l},\Ng}^{\textbf{e}}(1)$ is not negligible are thus those satisfying for any $1\leq i \leq K$ (with $K=\frac N 2$), whether 
\begin{align}\label{r1}
2+\alpha_i=\alpha_{i+1} \textit{ and } N_{\alpha_i}=N_{\alpha_{i+1}}=1,
\end{align}
or
\begin{align}\label{r2}
1+\alpha_i=\alpha_{i+1} \textit{ and } N_{\alpha_i}=2.
\end{align}
We call a couple satisfying one of these two relations an \textbf{admissible} couple.
For such a couple $(\textbf{e},\textbf{N})$, we denote by $J_2$ the set of indexes $i$ satisfying relation \eqref{r1} and $J_1$ those satisfying relation \eqref{r2}. For $\textbf{l}\in \mathcal G(\textbf{e})$, $b_{a_0,0,\epsilon,\textbf{l},\Ng}^{\textbf{e}}(1)$ can be decomposed by applying
 \eqref{magik} while noticing that the terms $b_{0,\ls_{i-1},\epsilon,(l_i,l_{i+1}),(1,1)}^{(0,\textbf{e})}$ are independent of the first index 
 $a_0$ (since $Q_{m,a}^{(0)}(\cdot)$ does not depend of $a$),
\begin{align*}
b_{0,0,\epsilon,\textbf{l},\Ng}^{\textbf{e}}(1)=\left(\prod_{i \in J_2}b_{0,\ls_{i-1},\epsilon,(l_i,l_{i+1}),(1,1)}^{(0,1)}(1)\right) \left(\prod_{i \in J_1}b_{a_0,\ls_{i-1},\epsilon,(l_i),(2)}^{(0)}(1)\right),
\end{align*}
with the convention $\ls_0=0$.
We recall that $g_{k_\epsilon}$ has zero mean (by Hypothesis \ref{H2}), thus $Q_{l_{i+1},a}^{(0)}(g_{k_\epsilon})=0$ and we can apply the following relation :
\begin{align*}
&b_{0,\ls_{i-1},\epsilon,(l_i,l_{i+1}),(1,1)}^{(0,1)}(1)=\frac{1}{(2\pi \Sigma l_i)^{1/2}}\sum_{a_1 \in \Z} b_{0,\ls_i,\epsilon,(l_{i+1}),(1)}(g_{k_\epsilon}(w_{\epsilon \ls_i}(x),.,a_1))\\
&=\sum_{l_{i+1}\leq \lfloor t/\epsilon\rfloor} \frac{1}{(2\pi \Sigma l_i)^{1/2}}\\
&\sum_{a_1,a_2 \in \mathbb Z}E_{\pmu}\left(  g_{k_\epsilon}(w_{\epsilon \ls_i}(x),.,a_1)1_{\{S_{l_{i+1}}=a_1-a_{2}\}}\circ \pt^{k_\epsilon} g_{k_\epsilon}(w_{\epsilon \ls_{i+1}}(x),\pt^{l_{i+1}}(\cdot),a_2) \right).\\
\end{align*}

To gather the last terms $b_{0,0,\epsilon,\textbf{l},\Ng}^{\textbf{e}}(1)$ and make a direct computation of the moments, we pull back the expression onto the dynamical system $(M,T,\mu)$ using the following lemma :  

\begin{lem}\label{convv}
For any $x_1, x_2 \in \mathbb{R}^d$,
$$
I_{\mu}\left(\F(x_1,T^l(\cdot))\F(x_2,\cdot)\right)=O(l^{-(1+\e/4)})
$$
uniformly in $x_1$ and $x_2$. In particular,
$$
\sigma^2(f)(x_1,x_2):=\sum_{l\in \Z}I_{\mu}\left(\F(x_1,T^{|l|}(\cdot))\F(x_2,\cdot)\right)
$$
converges and
\begin{align}\label{cruc}
\lim_{\epsilon \rightarrow 0} &\sum_{0\leq l \leq \lfloor t/\epsilon \rfloor}\sum_{a_1,a_2 \in \mathbb Z}E_{\pmu}\left(  g_{k_\epsilon}(w_{\epsilon (\ls_{i})}(x_1),\pt^{l}(\cdot),a_1)1_{\{S_{l}=a_1-a_2\}}\circ \pt^{k_\epsilon} g_{k_\epsilon}(w_{\epsilon (\ls_{i}+l)}(x_2),.,a_2) \right) \nonumber\\
&=\lim_{\epsilon\rightarrow 0} \sum_{l\geq 0}I_{\mu}(\F(w_{\epsilon (\ls_{i}+l)}(x_1),T^l(\cdot))\F(w_{\epsilon (\ls_i)}(x_2),\cdot)).
\end{align}

\end{lem}

\begin{proof}
To prove the first point, we can notice that the proof of Lemma \ref{aproxpro} actually states that for any $l \in \N^*$
\begin{align*}
\sup_{x_1,x_2\in \mathbb R}\left|I_{\mu}\left(\F(x_1,T^l(\cdot))\F(x_2,\cdot)\right)-
\sum_{a_1,a_2 \in \mathbb Z}E_{\pmu}\left(  g_{k_\epsilon}(x_1,\pt^{l}(\cdot),a_1)1_{\{S_{l}=a_1-a_2\}}\circ \pt^{k_\epsilon} g_{k_\epsilon}(x_2,.,a_2) \right)\right|\\
=O(l^2\beta_0^{k_\epsilon}\|f\|_{\infty}^2).
\end{align*}
Instead of  approximating $f$ by $g_{k_\epsilon}$ in the above formula, we approximate it by $g_{k_l'}$ with \\
$k_l':=\lceil \log(l)^2 \rceil$, then the difference above is summable and finally we only need to check the following convergence to ensure the summability of $\sup_{x_1,x_2\in \mathbb R}I_{\mu}\left(\F(x_1,T^l(\cdot))\F(x_2,\cdot)\right)$:
\begin{align}\label{eqconvappro}
\lim_{\epsilon \rightarrow 0} \sum_{l \leq \lfloor t/\epsilon \rfloor}\sup_{x_1,x_2\in \mathbb R}\left|\sum_{a_1,a_2 \in \mathbb Z}E_{\pmu}\left(  g_{k'_l}(x_1,\pt^{l}(\cdot),a_1)1_{\{S_{l}=a_1-a_2\}}\circ \pt^{k'_l} g_{k_{l}'}(x_2,.,a_2) \right)\right| < \infty.
\end{align}
In addition the rate of convergence of the terms in the above sum will provide the rate of convergence of the terms $I_{\mu}\left(\F(x_1,T^l(\cdot))\F(x_2,\cdot)\right)$.
We thus apply lemmas \ref{anisop} and \ref{spectre} page \pageref{spectre} to estimate the terms in the sum above: for any $x_1,x_2\in \mathbb R$,
\begin{align}
&\left|\sum_{a_1,a_2 \in \mathbb Z}E_{\pmu}\left(  g_{k'_l}(x_1,\pt^{l}(\cdot),a_1)1_{\{S_{l}=a_1-a_2\}}\circ \pt^{k_l'} g_{k_l'}(x_2,.,a_2) \right) \right|\nonumber \\
&=\left| \sum_{a_1,a_2 \in \mathbb Z}E_{\pmu}\left[ \sum_{u \in \Z}\sum_{v \in \Z}Q_{k_l',u}\left(g_{k_l'}(x_1,.,a_1)Q_{l-2k_l',a_2-a_1-u-v}\left(Q_{k_l',v}\left(P^{k_l'}g_{k_l'}(x_2,.,a_2)\right) \right) \right) \right] \right|\label{1eqconvsommoment} \\
&=\left|\sum_{a_1,a_2 \in \mathbb Z}E_{\pmu}\left[\sum_{u \in \Z}\sum_{v \in \Z} P^{k_l'}\left(Q_{k_l',u}(g_{k_l'}(x_1,.,a_1)G_{l-2k_l',a_2-a_1-u-v}\left(Q_{k_l',v}\left(P^{k_l'}g_{k_l'}(x_2,.,a_2)\right) \right) \right)\right]\right|\label{2eqconvsommoment}\\
&\leq \sum_{a_1,a_2 \in \mathbb Z}\sum_{u \in \Z}\sum_{|v|\leq Dk_l'}  \|P^{k_l'}Q_{k_l',u}(g_{k_l'}(x_1,.,a_1)\cdot)\|_{L(\B,\B)}\|G_{l-2k_l',a_2-a_1-u-v}\|_{L(\B,\B)}\|Q_{k_l',v}\left(P^{k_l'}(g_{k_l'}(x_2,.,a_2))\right)\|_{\B}\nonumber\\
&\leq (Dk_l')^2\sup_{x_1,x_2\in \mathbb R}\sum_{a_1,a_2 \in \mathbb{Z}} \|\F(x_1,.,a_2)\|_{\infty} \|\F(x_2,.,a_2)\|_{\B} O\left(\frac{\left((|a_2|+1)(|a_1|+1)(2k_l')\right)^{1+\e}}{(l-2k_l')^{1+\frac{\e} 2}}\right)\nonumber\\
&=O\left(\frac{(k_l')^{3+\e}}{(l-2k_l')^{1+\frac{\e} 2}}\right).\nonumber
\end{align}
Where $G_{n,l}$ is the operator defined in Proposition \ref{spectre}. and we went from line \eqref{1eqconvsommoment} to \eqref{2eqconvsommoment} by noticing that
$$
\sum_{a_2 \in \Z}Q^{(0)}_{l-k_l',a_2-a_1-u}\left(P^{k_l'}g_{k_l'}(x_2,.,a_2) \right)=\sum_{a_2 \in \Z}Q^{(0)}_{l-k_l',0}\left(P^{k_l'}g_{k_l'}(x_2,.,a_2) \right)=0,
$$ 
and the majoration at the last line is done using the Proposition \ref{spectre} on spectral  gap for perturbed operator and the assumption on the decay rate of $F$ from the Proposition \ref{propbirkhoff1} we want to prove :  $\sup_{x\in \mathbb R}\sum_{a \in \Z}(1+|a|)^{2(1+\e)}\|\F(x,.,a)\|_{\infty}<\infty$.
Thus the convergence in formula \eqref{eqconvappro} holds and we can conclude the first point of our lemma:
$$
\sup_{x_1,x_2\in \mathbb R}I_{\mu}\left(\F(x_1,T^l(\cdot))\F(x_2,\cdot)\right)=O(l^{-(1+\e/4)}) \textit{ and } \sum_{l \in \Z}\sup_{x_1,x_2\in \mathbb R}I_{\mu}\left(\F(x_1,\tp^{|l|}(\cdot))\F(x_2,\cdot)\right)<\infty.
$$
The above convergence along with the convergence stated in Lemma \ref{aproxpro} then provides the conslusion of the second point of the lemma: the convergence in formula \eqref{cruc}.
\end{proof}
From Lemma \ref{convv} we deduce that for $(\textbf{e},\textbf{N})$ satisfying \eqref{r1} or \eqref{r2}, 
\begin{align}
&\lim_{\epsilon \rightarrow 0}\epsilon^{N/4}\sum_{\textbf{l}\in \mathcal G(\textbf{e})}b_{a_0,0,\epsilon,\textbf{l},\Ng}^{\textbf{e}}(1)=\lim_{\epsilon \rightarrow 0} \epsilon^{N/4}\sum_{\textbf{l}\in \mathcal G(\textbf{e})}\left(\prod_{i \in J_2}b_{0,\ls_{i-1},\epsilon,(l_i,l_{i+1}),(1,1)}^{(0,1)}(1)\right) \left(\prod_{i \in J_1}b_{a_0,\ls_{i-1},\epsilon,(l_i),(2)}^{(0)}(1)\right)\nonumber\\
&= \lim_{\epsilon \rightarrow 0}\frac{\epsilon^{N/4}}{(2\pi \Sigma)^{N/4}}\sum_{\textbf{l}\in \mathcal G(\textbf{e})}\left(\prod_{i=1}^K\frac 1 {l_{\alpha_i}^{1/2}}\right)\nonumber\\
&\left(\prod_{i \in J_2}I_{\mu}\left(\F(w_{\epsilon\ls_i}(x),\cdot)\F(w_{\epsilon\ls_{i+1}}(x),T^{l_{i+1}}(\cdot)\right)\right) \left(\prod_{i \in J_1}I_{\mu}(\F(w_{\epsilon\ls_i}(x),\cdot)^2)\right)\nonumber\\
&= \lim_{\epsilon \rightarrow 0}\frac{\epsilon^{N/4}}{(2\pi \Sigma)^{N/4}}\sum_{1\leq l_1,\dots,l_{K},\ls_{K}\leq \lfloor t/\epsilon \rfloor}\left(\prod_{i=1}^K\frac 1 {l_i^{1/2}}\right)\label{eqegali}\\
&\left(\prod_{i \in J_2}\sum_{l=1}^\infty I_{\mu}\left(\F(w_{\epsilon\ls_i}(x),\cdot)\F(w_{\epsilon\ls_i}(x),T^{l}(\cdot)\right)\right) \left(\prod_{i \in J_1}I_{\mu}(\F(w_{\epsilon\ls_i}(x),\cdot)^2)\right),\label{egali}
\end{align}
where \eqref{egali} is obtained by noticing that Lemma \ref{convv} allows us to neglect some perturbations in the first coordinate of $f$ :
Indeed, fix the parameter $\alpha_\epsilon>0$ such that $\alpha_\epsilon \underset{\epsilon \to 0}{\to}\infty$ and $\alpha_\epsilon \underset{\epsilon \to 0}{=}o\left( \frac 1 \epsilon\right)$;
first one can split $\mathcal G(\textbf{e})=A\cup \bigcup_{i\in J_2}A_{i+1}$ between a central part\\
$A:=\{\textbf{l}\in \mathcal G(\textbf{e}),\sum_{j,j-1 \notin J_2}l_j\leq \lfloor t/\epsilon \rfloor-K\alpha_\epsilon\}$ and marginal ones composed by\\
$A_i:=\{\textbf{l}\in \mathcal G(\textbf{e}),\lfloor t/\epsilon \rfloor-\alpha_\epsilon \leq \sum_{j\neq i}l_j\leq \lfloor t/\epsilon \rfloor\}$.
Thanks to Lemma \ref{convv}, the terms in the sum in \eqref{eqegali} may be estimated from above by the term
\begin{align}\label{eqriemann}
\left(\prod_{i=1}^K\frac 1 {l_{\alpha_i}^{1/2}}\right)
\left(\prod_{i \in J_2}l_{i+1}^{-(1+\frac{\epsilon_0}{4})}\right) \left(\prod_{i \in J_1}I_{\mu}(\F(w_{\epsilon\ls_i}(x),\cdot)^2)\right).    
\end{align}
The convergence of Riemann sums then implies that the sum over each marginal term $A_i$ for $i-1\in J_2$\\
$\epsilon^{N/4}\sum_{\textbf{l}\in A_i}\left(\prod_{j=1}^K\frac 1 {l_{\alpha_j}^{1/2}}\right)
\left(\prod_{j \in J_2}l_{j+1}^{-(1+\frac{\epsilon_0}{4})}\right)$ tends to $0$. We are left with the sum over the central part $A$ which provides the line \eqref{egali} through the following reasoning :

indeed, in one hand \\
 $\sum_{(1/\epsilon)^{1/4} \leq l\leq \lfloor t/\epsilon \rfloor}  I_{\mu}\left(\F(w_{\epsilon\ls_i}(x),\cdot)\F(w_{\epsilon\ls_i+l}(x),T^{l}(\cdot))\right)$ converges to $0$ (uniformly in $l_i$) and on the other hand, since $s \mapsto  \F(w_s(x),\cdot)$  is uniformly Lipschitz, there is $C_{\F}>0$ such that
\begin{align*}
&\sum_{l \leq (1/\epsilon)^{1/4}}\left|I_{\mu}\left(\F(w_{\epsilon\ls_i}(x),\cdot)\F(w_{\epsilon\ls_i+l}(x),T^{l}(\cdot)\right)-I_{\mu}\left(\F(w_{\epsilon\ls_i}(x),\cdot)\F(w_{\epsilon\ls_i}(x),T^{l}(\cdot)\right)\right|\\
&\leq \sum_{l \leq (1/\epsilon)^{1/4}}I_{\mu}\left(|\F(w_{\epsilon\ls_i}(x),\cdot)| \right)C_{\F}\epsilon l=O(\epsilon^{1/2}),
\end{align*}
uniformly in $\textbf l$. 
Thus, we can substitute \eqref{egali}  \\
$\sum_{1\leq l_{i+1} \leq \lfloor t/\epsilon\rfloor -\ls_i}I_{\mu}\left(\F(w_{\epsilon\ls_i}(x),\cdot)\F(w_{\epsilon\ls_{i+1}}(x),T^{l_{i+1}}(\cdot))\right)$ by\\
$\sum_{1\leq l_{i+1} \leq \lfloor t/\epsilon\rfloor -\ls_i}I_{\mu}\left(\F(w_{\epsilon\ls_i}(x),\cdot)\F(w_{\epsilon\ls_{i+1}}(x),T^{l_{i+1}}(\cdot))\right)$\\  and since $\textbf{l}\in A$,  $\lfloor t/\epsilon\rfloor -\ls_i\geq \alpha_\epsilon \to \infty$  and thus it can be replaced by its limit\\
$\sum_{l=1}^\infty I_{\mu}\left(\F(w_{\epsilon\ls_i}(x),\cdot)\F(w_{\epsilon\ls_i}(x),T^{l}(\cdot))\right)$.

We introduce the following map $G$ corresponding to a continuous version of the terms in line \eqref{egali} with respect to $\textbf{u}\in \mathbb{R}_+^K$ instead of $\textbf{l}\in \N^K$.
\begin{lem}
The map $G$ described a follows is uniformly continuous on $\mathbb{R}_+^K$
\begin{align*}
G(\textbf{u}):=&\left(\prod_{i \in J_2}\sum_{l=1}^\infty I_{\mu}\left(\F(w_{\sum_{k=1}^iu_k}(x),.)\F(w_{\sum_{k=1}^iu_k}(x),T^{l}(.)\right)\right) \left(\prod_{i \in J_1}I_{\mu}(\F(w_{\sum_{k=1}^iu_k}(x),.)^2)\right).
\end{align*}
\end{lem}
\begin{proof}
Since the map $w$ is Lipschitz, it is enough to prove that $x \in \mathbb{R}^d \mapsto  \sum_{l=1}^\infty I_{\mu}\left(\F(x,.)\F(x,T^{l}(.))\right)$ is uniformly continuous. This is a consequence of the uniform convergence of $I_{\mu}\left(\F(x,.)\F(x,T^{l}(.)\right)$ in $O(l^{-(1+\e/4)})$ in Lemma \ref{convv} and the Lipschitz regularity of $\F(.,.)$ :

For any compact $K$, any $N\in \N^*$and any $x,y \in K$,
\begin{align*}
\left|\sum_{l=1}^\infty I_{\mu}\left(\F(x,.)\F(x,T^{l}(.))\right)-\sum_{l=1}^\infty I_{\mu}\left(\F(y,.)\F(y,T^{l}(.))\right)\right|\leq C_1 \sum_{|l|\geq N}|l|^{-1+\frac{\epsilon_0}4} + CN\|x-y\|.
\end{align*}
The latter inequation depends only on $\|x-y\|$ thus we deduce the uniform continuity
\end{proof}

Notice that for any $u,v \geq \epsilon$, $|\frac 1 u- \frac 1 v|=\left|\frac{u-v}{uv}\right|\leq \left(\frac 1 \epsilon\right)^{2} |u-v|$, thus the following Riemann sums converge :
\begin{align}
&\lim_{\epsilon \rightarrow 0}\epsilon^{K/2}\sum_{\textbf{l}\in (\N^*)^K,\ls_{K}\leq \lfloor t/\epsilon \rfloor}\left(\prod_{i=1}^K\frac 1 {l_i^{1/2}}\right)\nonumber\\
&\left(\prod_{i \in J_2}\sum_{l=1}^\infty I_{\mu}\left(\F(w_{\epsilon\ls_i}(x),\cdot)\F(w_{\epsilon\ls_i}(x),T^{l}(\cdot)\right)\right) \left(\prod_{i \in J_1}I_{\mu}(\F(w_{\epsilon\ls_i}(x),\cdot)^2)\right)\nonumber\\
&=\lim_{\epsilon \rightarrow \infty}\epsilon^{K}\sum_{\textbf{l},\ls_{K}\leq \lfloor t/\epsilon \rfloor}\int_{l_1-1<u_1\leq l_1}\dots \int_{l_K-1<u_K\leq l_K}\left(\prod_{i=1}^K\frac 1 {(\lceil u_i\rceil \epsilon)^{1/2}}\right)G(\epsilon \lceil u_1\rceil,\dots,\epsilon \lceil u_K\rceil)d\textbf{u}\nonumber\\
&=\lim_{\epsilon \rightarrow \infty}\int_{0<u_1,\dots,u_K\leq \lfloor t/\epsilon\rfloor, \sum_i \lceil u_i/\epsilon\rceil \leq \lfloor t/\epsilon\rfloor}\left(\prod_{i=1}^K\frac 1 {(\lceil u_i/\epsilon\rceil \epsilon)^{1/2}}\right)G(\epsilon \lceil u_1/\epsilon\rceil,\dots,\epsilon \lceil u_K/\epsilon \rceil)d\textbf{u} \label{adomine}\\
&=\int_{0< u_1,\dots,u_K, \sum_i u_i \leq t}\left(\prod_{i=1}^K\frac 1 {u_i^{1/2}}\right)G(u_1,\dots,u_n)du_1,\dots,du_K. \label{integ}
\end{align}
The last line is obtained through dominated convergence of the term \eqref{adomine} which is upper bounded by
$
C\left(\prod_{i=1}^K\frac {1_{\{0<u_i\leq t+1\}}}  { u_i^{1/2}}\right),
$
where $C>0$ comes from the boundedness of $G(\cdot)$ as a continuous function on a compact.

We denote by $\mathcal C (N,q)$ the set of \textbf{admissible} couple $(\Ng,\textbf e)$ of length $q$ such that $\underline{\Ng}_q=N$.
When $N$ is odd, we notice that $\mathcal C(N,q)=\emptyset$ thus
$$
\lim_{\epsilon \rightarrow 0}E_{\mup}\left( v_t^\epsilon(x,\cdot)^N \right)=0.
$$
We consider now the case where the integer $N$ is even,
\begin{align}\label{obscur2}
\lim_{\epsilon \rightarrow 0}&E_{\mup}\left( v_t^\epsilon(x,\cdot)^N \right)=\lim_{\epsilon \rightarrow 0} \epsilon^{N/4}\sum_{q=1}^N \sum_{N_i\geq 1, \sum_{i=1}^q N_i=N}C_{\Ng}A_{\epsilon,q,\Ng}(\F)\nonumber\\
&=\lim_{\epsilon \rightarrow 0} \epsilon^{N/4}\sum_{q=1}^N \sum_{(\textbf{e},\Ng)\in \mathcal C (N,q)}C_{\Ng}\sum_{l_1,\dots,l_q,\ls_q\leq \lfloor t/\epsilon \rfloor}b_{a_0,0,\epsilon,\textbf{l},\Ng}^{\textbf{e}}(1),
\end{align}
and so
\begin{align}
&\lim_{\epsilon \rightarrow 0}E_{\mup}\left( v_t^\epsilon(x,\cdot)^N \right)=\sum_{q=1}^N\sum_{(\Ng,\textbf e)\in \mathcal C (N,q)}C_{\Ng} \frac{1}{(2\pi \Sigma)^{N/4}}\int_{0\leq u_1,\dots,u_K\leq 1, \sum_i u_i \leq 1}\left(\prod_{i=1}^K\frac 1 {u_i^{1/2}}\right)\nonumber\\
& \left(\prod_{i \in J_2(\Ng)}\sum_{l=1}^\infty I_{\mu}\left(\F(w_{\sum_{k=1}^i u_k}(x),\cdot)\F(w_{\sum_{k=1}^i u_k}(x),T^{l}(\cdot)\right)\right) \left(\prod_{i \in J_1(\Ng)}I_{\mu}(\F(w_{\sum_{k=1}^iu_k}(x),\cdot)^2)\right)d\textbf{u}.\nonumber
\end{align}
In that case $q$, $\Ng \in \mathcal C (q)$, $C_{\Ng}$ et $J_1(\Ng)$
Notice that given $N$, the set $J_2(\Ng)$ determines entirely the quantities $q$, $(\Ng,\textbf e)\in \mathcal C(N,q)$ and $J_1(\Ng)$ :\\
$N/2=|J_1|+|J_2|$, $ q=|J_1|+2|J_2|$ and $\psi : (\Ng,\textbf e) \in \mathcal C (N,q) \mapsto  J_2(\Ng)$ is injective.\\
Notice also that by recurrence on the cardinality of $J_2(\Ng)$, $C_{\textbf N}=2^{|J_2|}\frac{N!}{2^{N/2}}$.
\begin{align*}
&\lim_{\epsilon \rightarrow 0}E_{\mup}\left( \tilde v_t^\epsilon(x,\cdot)^N \right)=\sum_{J_2 \in \Pa([|1;N/2|])}\frac{(2\pi \Sigma)^{-N/4}N!}{2^{N/2}(N/2)!}2^{|J_2|}(N/2)!\int_{0\leq u_1,\dots,u_{N/2}, \sum_i u_i \leq t}\left(\prod_{i=1}^{N/2}\frac 1 {u_i^{1/2}}\right)\\
& \left(\prod_{i \in J_2(\Ng)}\sum_{l=1}^\infty I_{\mu}\left(\F(w_{\sum_{k=1}^i u_k}(x),\cdot)\F(w_{\sum_{k=1}^i u_k}(x),T^{l}(\cdot)\right)\right)\left(\prod_{i \in J_1(\Ng)}I_{\mu}(\F(w_{\sum_{k=1}^iu_k}(x),\cdot)^2)\right)d\textbf{u}\\
\end{align*}
\begin{align*}
&=\frac{(2\pi \Sigma)^{-N/4}N!}{2^{N/2}(N/2)!}\sum_{J_2(\Ng) \in \Pa([|1;N/2|])}(N/2)!\int_{0\leq u_1,\dots,u_{N/2}, \sum_i u_i \leq t}\left(\prod_{i=1}^{N/2}\frac 1 {u_i^{1/2}}\right)\\
& \left(\prod_{i \in J_2(\Ng)}2\sum_{l=1}^\infty I_{\mu}\left(\F(w_{\sum_{k=1}^i u_k}(x),\cdot)\F(w_{\sum_{k=1}^i u_k}(x),T^{l}(\cdot)\right)\right) \left(\prod_{i \in J_1(\Ng)}I_{\mu}(\F(w_{\sum_{k=1}^iu_k}(x),\cdot)^2)\right)d\textbf{u}\\
&=\frac{N!}{2^{N/2}(N/2)!} (2\pi \Sigma)^{-N/4}(N/2)!\\
&\int_{0\leq u_1,\dots,u_{N/2}, \sum_i u_i \leq t}\left(\prod_{i=1}^{N/2}\frac 1 {u_i^{1/2}}\right)\sum_{l \in \Z} I_{\mu}\left(\F(w_{\sum_{k=1}^i u_k}(x),\cdot)\F(w_{\sum_{k=1}^i u_k}(x),T^{|l|}(\cdot)\right)d\textbf{u}.\\
&=\frac{N!}{2^{N/2}(N/2)!} (2\pi \Sigma)^{-N/4}(N/2)!\\
&\int_{\textbf{u}\in \mathbb R_+^{\frac N 2}, \sum_i u_i \leq t}\left(\prod_{i=1}^{N/2}\frac 1 {u_i^{1/2}}\right)\sum_{l \in \Z} I_{\mu}\left(\F(w_{\sum_{k=1}^i u_k}(x),\cdot)\F(w_{\sum_{k=1}^i u_k}(x),T^{|l|}(\cdot)\right)d\textbf{u}  
\end{align*}

The term in $\frac{N!}{2^{N/2}(N/2)!}$ match with the moment of order $N/2$ of a standard normal law (which is the $\frac N 2$th moment of $B_1$ for $(B_t)_{t\geq 0}$ Brownian motion of law of variance $1$). The second factor can be identified using \cite[ proposition B.1.17]{MPphd}  with the moment of order $N$ of the variable $\left(\int_0^t a(w_s(x))dL_s'(0)\right)^{1/2}$ 
  with $a(x)$ given by~\eqref{formulea}
  and $(L_s')_{s\geq 0}$ a local time process $(L_s'(0))_{t\geq 0}$ of a Brownian motion  $(B'_t)_{t\geq 0}$ of variance $\Sigma$. The limit law is thus the law of a Gaussian vector conditionally to the local time process $(L_s'(0))_{t\geq 0}$ (of a Brownian motion  $(B'_t)_{t\geq 0}$ of variance $\Sigma$ )\footnote{with variance $\left(\int_0^t \langle u,a(w_s(x))u\rangle dL_s'(0)\right)^{1/2}$ when treating the dimension $d>1$}. We now check that this limit law corresponds to the law of $\int_0^t \sqrt{a(w_s(x))}dB_{L_s'(0)}$ from the conclusion of the theorem \ref{thmbirkhoff} we want to prove\footnote{when $d>1$, we check in the same way that conditionally to the process $(L_s'(0))_{s\geq 0}$, the marginals of $\langle u,\int_0^t \sqrt{a(w_s(x))}dB_{L_s'(0)}\rangle$ are Gaussian and of variance $\int_0^t \langle u,a(w_s(x))u\rangle dL_s'(0)$.}.
   To do so we only need to prove that $\int_0^t \sqrt{a(w_s(x))}dB_{L_s'(0)}$ is a Gaussian vector conditionally to  $(L_s(0))_{t\geq 0}$ and compute its variance.
We fix $(L_s'(0))_{s\geq 0}:=(h(s))_{s\geq 0}$, since $B$ is independent of that process, its law is still driven by the law of the process $(B_{h(s)})_{s\geq 0}$ which is then a martingale with quadratic variation $\langle B_{h(s)}\rangle=h(s)$. Thus Dubins-Schwarz relation \cite{DuSC65} ensures that the law of the martingale  $\left(\int_0^t \sqrt{a(w_s(x))}dB_{L_s'(0)}\right)_{t\geq 0}$ matches with the law of  the process  $\left(\beta_{\int_0^t a(w_s)df(s)}\right)_{t\geq 0}$ with $\beta$ a standard Brownian motion. In particular, the law of the random variable $\int_0^t \sqrt{a(w_s(x))}dB_{L_s'(0)}$ corresponds to the Gaussian law of $\beta_{\int_0^t a(w_s)df(s)}$ whose variance is $\left(\int_0^t a(w_s)df(s)\right)^{1/2}$. Thus we deduce the convergence of moments, for any $N\in \N$ :
\begin{align*}
\lim_{\epsilon \rightarrow 0}E_{\mup}\left( v_t^\epsilon(x,\cdot)^N \right)=E\left(\left(\mathcal N\left(\int_0^t a(w_s(x))dL_s(0)\right)^{1/2}\right)^N\right).
\end{align*}

We deduce from the convergence above of moments of  $v_t^\epsilon(x,\cdot)$ to those of 
$\mathcal N\left(\int_0^t a(w_s(x))dL_s(0)\right)^{1/2}$ the corresponding convergence in law using Carleman's criterion (see \cite{carlman}) whose assumptions are checked in the following Lemma \ref{lemcarlmantplc}.
\end{proof}

\begin{lem}\label{lemcarlmantplc}
For any $t \in [0,S]$, the sequence $(M_N)_{N\in \N}$ of moments of $\int_0^t \sqrt{a(w_s(x))}dB_{L_s'(0)}$  defined by $M_N:=E\left(\left(\int_0^t \sqrt{a(w_s(x))}dB_{L_s'(0)}\right)^N\right)$ satisfies Carleman's criterion, i.e $\sum_{N\geq 0}M_{2N}^{-1/(2N)}$ diverges (see \cite{carlman}).
\end{lem}

\begin{proof}
We deduce the moments of order $2N$ from the independence between the process $(L_s)_{s\geq 0}$ and the gaussian law $\mathcal N$ :
\begin{align}
E\left(\left(\mathcal N\left(\int_0^t a(w_s(x))dL_s(0)\right)^{1/2}\right)^N\right)&=E\left(\mathcal N^{2N}\right)E\left(\left(\int_0^t a(w_s(x))dL_s(0)\right)^N\right)\nonumber\\
&\leq \frac{(2N)!}{2^{N}N!}\|a(\cdot)\|_{\infty,w_{[0,T]}(x)}E(L_T^N).\label{carlamn2}
\end{align}
Stirling formula states that  $(2N)!=O\left(\frac{(2N)^{2N+1}}{e^{2N}}\right)$ and thus $\left((2N)!\right)^{\frac{1}{2N}}=O(N)$ whereas the last term \eqref{carlamn2} is in $O(N)$ according to the technical computation in \cite[Corollaire 1.5.5]{MPphd}.
thus $M_N^{\frac{1}{2N}}$ is in order $O(N)$ and thus its inverse is the general term of a divergent series.
\end{proof}

\section{Proof of theorem \ref{thmbirkhoff} : finite dimensional distributions.}\label{sectfinidimbirk}

We now focus on proving the following lemma consisting on the convergence of the finite dimensional distribution (f.d.d) in the conclusion of theorem \ref{thmbirkhoff}. 
\begin{lem}
For any $x\in \mathbb R$, the process $(v_t^\epsilon(x,\cdot))_{t \in [0,T]}$ converges in law for the f.d.d (with respect to measure $\mup$) to the process $\left(\int_0^t \sqrt{a(w_s(x))}dB_{L_s'(0)}\right)_{t\in [0,S]}$,
\end{lem}

\begin{proof}
Let $x \in \mathbb{R}$,
It is enough to prove the convergence of $(\tilde v_t^\epsilon(x,\cdot))_{t\in [0,S]}$. Indeed since $\F$ is uniformly Lipschitz and bounded in its first coordinate and $(w_s(x))_s$ is a Lipschitz map,
\begin{align*}
&\sup_{t \in [0,T]}\|v_t^\epsilon(x,\cdot)-\tilde v_t^\epsilon(x,\cdot)\|_{L^1_{\mup}}\\
&=\epsilon^{1/4}\sup_{t \in [0,T]}\left\|\sum_{k=1}^{\lfloor t/\epsilon \rfloor} \int_{s=k-1}^{k}f(w_{\epsilon k}(x),T^{k}\cdot)-f(w_{\epsilon s}(x),T^{\lfloor s \rfloor}\cdot)ds +\int_{\lfloor t/\epsilon \rfloor}^{t/\epsilon}f(w_{\epsilon s}(x),T^{\lfloor s \rfloor}\omega)ds\right\|_{L^1_{\mup}}\\
&\leq \epsilon^{1/4}(TC_f+\|f\|_{\infty}).
\end{align*}
Thus $(v_t^\epsilon(x,\cdot))_{t\in [0,S]}$ converges in law iff $(\tilde v_t^\epsilon(x,\cdot))_{t\in [0,S]}$ converges.
Let $0=t_0<t_1<\dots<t_p \in [0,S]$, by linear combinations, it is enough to prove that $\left(\epsilon^{1/4}\tilde{v}^{\epsilon}_{t_i}(x,\cdot)-\epsilon^{1/4}\tilde{v}^{\epsilon}_{t_{i-1}}(x,\cdot)\right)_{i=1}^p$ converges to $\left(\int_{t_{i-1}}^{t_i} \sqrt{a(w_s(x))}dB^{(i)}_{L_s'(0)}\right)_{i=1,\dots,p}\overset{\mathcal{L}}{=}\left(\int_{0}^{t_i} \sqrt{a(w_s(x))}dB_{L'_s(0)}-\int_{0}^{t_{i-1}} \sqrt{a(w_s(x))}dB_{L'_s(0)}\right)_{i=1,\dots,p}$
where the random processes $B^{(i)}$ are mutually independent standard motions independent from $\left(B_t'\right)_{t \in [0,T]}$.\\
Carleman criterion applies to these random variable (see Lemma \ref{lemcarlmantplc} page \pageref{lemcarlmantplc}) thus we only need to check the convergence of the moments of $\left(\epsilon^{1/4}\tilde{v}^{\epsilon}_{t_i}(x,\cdot)-\epsilon^{1/4}\tilde{v}^{\epsilon}_{t_{i-1}}(x,\cdot)\right)_{i=1}^p$ through the following Lemma \ref{lemmomentvt}.

\end{proof}
\begin{lem}\label{lemmomentvt}
For any\footnote{When $d>1$  the convergence of moment it is enough to prove the limit in Lemma \ref{lemmomentvt} identity for
$\lim_{\epsilon \rightarrow 0} E_{\mup}\left( \prod_{i=1}^p \langle \theta_i, \tilde{v}^{\epsilon}_{t_i}(x,\cdot)-\tilde{v}^{\epsilon}_{t_{i-1}}(x,\cdot)\rangle^{M_i}\right)=E\left( \prod_{i=1}^p \langle \theta_i,\int_{t_{i-1}}^{t_i} \sqrt{a(w_s(x))}dB^{(i)}_{L_s'(0)}\rangle^{M_i}\right)$
for any $\theta_i \in \mathbb R^d$} $M_1,\dots,M_p \geq 0$,
$$
\lim_{\epsilon \rightarrow 0} E_{\mup}\left( \prod_{i=1}^p \left(\tilde{v}^{\epsilon}_{t_i}(x,\cdot)-\tilde{v}^{\epsilon}_{t_{i-1}}(x,\cdot) \right)^{M_i}\right)=E\left( \prod_{i=1}^p \left(\int_{t_{i-1}}^{t_i} \sqrt{a(w_s(x))}dB^{(i)}_{L_s'(0)}\right)^{M_i}\right).
$$
\end{lem}

\begin{proof}
The term
$E_{\mup}\left( \prod_{i=1}^p \left(\tilde{v}^{\epsilon}_{t_i}(x,\cdot)-\tilde{v}^{\epsilon}_{t_{i-1}}(x,\cdot) \right)^{M_i}\right)$ can be expanded as follows,
\begin{align}\label{finidib}
E_{\mup}&\left( \prod_{i=1}^p \left(\tilde{v}^{\epsilon}_{t_i}(x,\cdot)-\tilde{v}^{\epsilon}_{t_{i-1}}(x,\cdot) \right)^{M_i}\right)=\nonumber\\
&\sum_{q_1=1}^{M_1}\dots\sum_{q_p=1}^{M_p}\sum_{1 \leq N_1,\dots,N_{\qs_p},M_i=\sum_{j=\qs_{i-1}+1}^{\qs_i}N_j}\left(\prod_{i=1}^pC_{(N_{\qs_{j-1}+1},\dots,N_{\qs_j})}\right)A_{\epsilon,\textbf q,\textbf N};
\end{align}
with $\qs_i:=\sum_{j=1}^iq_j$, as in Proposition \ref{propbirkhoff1} and  
\begin{align*}
A_{\epsilon,\textbf q,\textbf N}(F)&=\sum_{1 \leq n_1 <\dots < n_{q_1} \leq \lfloor  t_1/\epsilon \rfloor}\dots \sum_{\lfloor  t_{p-1}/\epsilon \rfloor < n_{\qs_{p-1}+1} <\dots < n_{\qs_p} \leq \lfloor  t_p/\epsilon \rfloor}\\
&\sum_{\textbf a \in \mathbb{Z}^{\qs_p}}E_{\mup}\left( \prod_{i=1}^{\qs_p} \F(w_{\epsilon n_i}(x),\tp^{n_i}(\cdot),a_i)^{N_i}1_{\{S_{n_i}=a_i\}} \right)\\
&=\sum_{\textbf l \in E_\epsilon^{\textbf q}(t_1,\dots,t_p)}b_{0,0,\epsilon,\textbf{l},\textbf{N}}(1),
\end{align*}
where\footnote{when $d>1$, we replace $F$ in the formula \eqref{finidib} by $\F_i=\langle\theta_j,\F\rangle$ for all $i$ and $j$ such that $1+\qs_{j-1} \leq i \leq \qs_j$.} for any map $f$ satisfying the same assumptions as the map $F$ in Theorem \ref{thmbirkhoff},
\begin{align*}
&b_{a_0,m,\epsilon,\textbf{l},\Ng}(f):=\sum_{\textbf a \in \mathbb{Z}^{\qs_p}}E_{\mup}\left(f\left( \prod_{i=1}^{\qs_p} f(w_{\epsilon^{-1} (\ls_i+m)}(x),\tp^{l_i}(\cdot),a_i)^{N_i}1_{\{\bar S_{l_i}=a_i-a_{i-1}\}}\circ \tp^{k_\epsilon} \right)\circ \tp^{\ls_{i-1}} \right),
\end{align*}
and
$$
E_\epsilon^{\textbf q}(t_1,\dots,t_p):=\{\textbf{l}, l_i \geq 1 : \lfloor  t_{j-1}/\epsilon \rfloor < \ls_i \leq \lfloor  t_j/\epsilon \rfloor, \forall i, 1+\qs_{j-1} \leq i \leq \qs_j, \forall j, 1 \leq j \leq p \}.
$$
 We denote $N:=\sum_{i=1}^pM_i$. 
According to equation \eqref{obscur}, the \textbf{admissible} couples $(\textbf{N},\textbf{e})$ are among those satisfying relations \eqref{r1} or \eqref{r2} at page \pageref{r1}. So for some fixed $\textbf{e}$, the set $\mathcal C (N,\sum_{i=1}^pq_i)$ (introduced page \pageref{integ}) of \textbf{admissible} couples satisfies the following inclusion, 
$\mathcal C (N,\sum_{i=1}^pq_i)=\mathcal C (M_1,q_1)\times \dots \times \mathcal C (M_p,q_p)\cup \bigcup_{j=2}^{p-1} A_j $ with\\ 
$A_j:=\{(\textbf{N},\textbf{e})\in \mathcal{C}(N,\sum_{i=1}^pq_i),  N_{\qs_j}=N_{1+\qs_{j}}=1, \, e_{\qs_j}=0, \, e_{\qs_{j}+1}=1\}$.
We now prove that for any couple  $(\textbf{N},\textbf{e})\in A_j$, $\epsilon^{N/4}A_{\epsilon,\textbf q,\textbf N}(F)$ is negligible (i.e $A_j$ makes no contribution). Fix $j\in \{2,\dots,p-1\}$ and  $(\textbf{N},\textbf{e})\in A_j$, then according to equation  \eqref{egali} where we introduced for $k\in \{1,\dots,\frac N 2\}$ notation  $\alpha_k$ to represent the indexes such that $e_{\alpha_k}=0$,
\begin{align}
\epsilon^{N/4}\sum_{\textbf l \in E_\epsilon^{\textbf q}(t_1,\dots,t_p)}|b^{\textbf{e}}_{0,0,\epsilon,\textbf{l},\textbf{N}}(1)|&\leq \epsilon^{N/4} \sum_{\textbf l \in E_\epsilon^{\textbf q}(t_1,\dots,t_p)}\left(\prod_{k=1}^{N/2}l_{\alpha_k}^{-1/2}\right)\label{pata1}\\
&\left(\prod_{i \in J_2(\textbf N)}\left| I_{\mu}\left(\F(w_{\epsilon\ls_{i}}(x),\cdot)\F(w_{\epsilon\ls_{i+1}}(x),T^{l_{i+1}}(\cdot)\right)\right|\right)\label{pata2}\\
&\left(\prod_{i \in J_1(\textbf{N})}\left|I_{\mu}\left(\F(w_{\epsilon\ls_i}(x),\cdot)^2\right)\right|\right)\label{primun}
\end{align}
From Lemma \ref{convv} page \pageref{convv} we deduce that the terms \eqref{pata1}, \eqref{pata2} and \eqref{primun} are bounded. 
Notice that by definition of $A_j$,  $\qs_j \in J_2(\bf N)$ and  $\ls_{\qs_j}\leq  t_j/\epsilon <\ls_{1+\qs_j}$ . Thus, having $e_{1+\qs_j}=1$ means that \\\textbf{either}  $l_{1+\qs_j}\geq k_\epsilon$ and then the marginal sum $\sum_{l_{1+\qs_j}\geq k_\epsilon}\left| I_{\mu}\left(\F(w_{\epsilon\ls_{i}}(x),\cdot)\F(w_{\epsilon\ls_{i+1}}(x),T^{l_{i+1}}(\cdot)\right)\right|$ tends  to $0$ when $\epsilon$ tends to $0$ and so does $\epsilon^{N/4}A_{\epsilon,\textbf q,\textbf N}(\F)$,  \textbf{or}  $t_j/\epsilon-\ls_{\qs_j}\leq  k_\epsilon$ in which case

 the upper estimate of the terms given in formula \eqref{eqriemann} applies to the terms in the marginal sums over the elements  $\textbf{l}\in E_\epsilon^{\textbf q}(t_1,\dots,t_p)$ such that $ t_j/\epsilon -\ls_{\qs_j}\leq k_\epsilon$ in \eqref{pata1}. Then these marginal sums converge as a Riemann approximations to $0$. Thus the term  $\epsilon^{N/4}A_{\epsilon,\textbf q,\textbf N}(F)$ is negligible and the contribution of the couple $(\textbf{N},\textbf{e})\in A_j$  is negligible; the set of \textbf{admissible} couple $\mathcal C (N,\sum_{i=1}^pq_i)$ can be identified with $\mathcal C (M_1,q_1)\times \dots \times \mathcal C (M_p,q_p)$.
In particular, if there is $i\leq p$ such that $M_i$ is odd then \eqref{finidib} tends to $0$ which is consistent with the fact that
$$
E\left( \prod_{i=1}^p \left(\int_{t_{i-1}}^{t_i} \sqrt{a(w_s(x))}dB^{(i)}_{L_s'(0)}\right)^{M_i}\right)=0.
$$
We recall that $\left((B^{(i)}_t)_{t\geq 0}\right)_{i \in \N}$ are mutually independent and independent of $(L'_t)_{t\geq 0}$, with symmetric law.
 We suppose now that $M_1,\dots, M_p$ are all even and choose\\
 $(\textbf{N},\textbf{e}) \in \mathcal C (M_1,q_1)\times \dots \times \mathcal C (M_p,q_p)$. 
 Equation \eqref{integ}  page \pageref{integ} implies,  \\
 
\begin{align*}
&\lim_{\epsilon \rightarrow 0} \epsilon^{N/4}A_{\epsilon,\textbf q,\textbf N}(F)=\lim_{\epsilon \rightarrow 0} \epsilon^{N/4} \left(\prod_{j=1}^p\frac{1}{(2\pi \Sigma)^{M_j/4}}\right)\sum_{\textbf l \in \tilde E_\epsilon^{\textbf q}(t_1,\dots,t_p)}\left(\prod_{i=1}^{N/2}\frac 1 {l_i^{1/2}}\right)\nonumber\\
&\left(\prod_{i \in J_2(\textbf N)}\sum_{l=1}^\infty I_{\mu}\left(\F(w_{\epsilon \ls_i}(x),\cdot)\F(w_{\epsilon\ls_i}(x),T^{l}(\cdot)\right)\right) \left(\prod_{i \in J_1(\textbf N)}I_{\mu}(\F(w_{\epsilon\ls_i}(x),\cdot)^2)\right)\\
&=\left(\prod_{j=1}^p\frac{1}{(2\pi \Sigma)^{M_j/4}}\right)\int_{F_{M_1,\dots,M_p}}\left(\prod_{i=1}^{N/2}\frac 1 {u_i^{1/2}}\right)\nonumber\\
&\left(\prod_{i \in J_2(\textbf N)}\sum_{l=1}^\infty I_{\mu}\left(\F(w_{t_{j-1}+\sum_{k=1}^iu_k}(x),\cdot)\F(w_{t_{j-1}+\sum_{k=1}^iu_k}(x),T^{l}(\cdot)\right)\right)\\
& \left(\prod_{i \in J_1(\textbf N)}I_{\mu}(\F(w_{t_{j-1}+\sum_{k=1}^iu_k}(x),\cdot)^2)\right)d\textbf{u},
\end{align*}
Where  $\tilde E_\epsilon^{\textbf q}(t_1,\dots,t_p):=\{1\leq l_1,\dots,l_{\frac{M_1+\dots+M_p} {2}} : \lfloor  t_{j-1}/\epsilon \rfloor<\ls_k \leq \lfloor  t_j/\epsilon \rfloor, \forall k, \frac{M_1+\dots+M_{j-1}} {2}+1\leq k \leq \frac{M_1+\dots+M_j} {2} \}$ and
$F_{M_1,\dots,M_p}:=$\\
$\{1\leq u_1,\dots,u_{\frac{M_1+\dots+M_p} {2}} :  \lfloor t_{j-1}/\epsilon \rfloor<\underline u_k \leq \lfloor t_j/\epsilon \rfloor, \forall k, \frac{M_1+\dots+M_{j-1}} {2}+1\leq k \leq\frac{M_1+\dots+M_j} {2} \}$.\\
We introduce the notation $\mathcal C' (\textbf{M},q)$ for the non negligible \textbf{admissible} couples, \\
$\mathcal C' (\textbf{M},q):=\mathcal C (M_1,q_1)\times \dots \times C (M_p,q_p) $, then  relation \eqref{finidib} sums up to
\begin{align*}
&\lim_{\epsilon \rightarrow 0}E_{\mup}\left( \prod_{i=1}^p \left(\tilde{v}^{\epsilon}_{t_i}(x,\cdot)-\tilde{v}^{\epsilon}_{t_{i-1}}(x,\cdot) \right)^{M_i}\right)=\\
&\sum_{q_1=1}^{M_1}\dots\sum_{q_p=1}^{M_p}\sum_{(\textbf{N},\textbf{e})\in \mathcal C' (\textbf{M},q)}\left(\prod_{i=1}^pC_{(N_{\qs_{i-1}+1},\dots,N_{\qs_i})}\right)\int_{F_{M_1,\dots,M_p}}\frac{1}{(2\pi \Sigma)^{M_i/2}}\left(\prod_{i=1}^{N/2}\frac 1 {u_i^{1/2}}\right)\nonumber\\
&\left(\prod_{i \in J_2(\textbf N)}\sum_{l=1}^\infty I_{\mu}\left(\F(w_{t_{j-1}+\sum_{k=1}^iu_k}(x),\cdot)\F(w_{t_{j-1}+\sum_{k=1}^iu_k}(x),T^{l}(\cdot)\right)\right)\\
& \left(\prod_{i \in J_1(\textbf N)}I_{\mu}(\F(w_{t_{j-1}+\sum_{k=1}^iu_k}(x),\cdot)^2)\right)d\textbf{u}.
\end{align*}
The same reasoning from formulas \eqref{obscur2} page \pageref{obscur2}, gives still with the even quantity $N:=\sum_{i=1}^pM_i$,
\begin{align}
\lim_{\epsilon \rightarrow 0} E_{\mup}\left( \prod_{i=1}^p \left(\tilde{v}^{\epsilon}_{t_i}(x,\cdot)-\tilde{v}^{\epsilon}_{t_{i-1}}(x,\cdot) \right)^{M_i}\right)&=\sum_{q_1=1}^{M_1}\dots\sum_{q_p=1}^{M_p}\prod_{i=1}^p\frac{M_i!}{2^{M_i/2}(M_i/2)!(2\pi \Sigma)^{M_i/4}}\nonumber\\
&\left( \frac {M_i} 2\right)!\int_{F_{M_1,\dots,M_p}}\frac{\prod_{j=1}^{N/2}a(w_{\sum_{i=1}^ju_i}(x))}{\prod_{j=1}^{N/2}u_j^{1/2}}d\textbf{u}\label{eqprodmoment}\\
&=\prod_{i=1}^pE(B_1^{M_i})E\left( \prod_{i=1}^p\left(\int_{t_{i-1}}^{t_i}a(w_s(x))dL_s'(0)\right)^{M_i/2}\right)\nonumber\\
&=E\left(\prod_{i=1}^p\left(\int_{t_{i-1}}^{t_i} \sqrt{a(w_s(x))}dB^{(i)}_{L_s'(0)}\right)^{M_i}\right)\label{eqprodmoment2}
\end{align}

Equation \eqref{eqprodmoment} corresponds to a product of moment of independent variable. We refer to  \cite[Appendix B.1.2]{MPphd} for the explicit computation of the moments of an integral driven by a Local time. This leads to the identification made at line\footnote{When $d>1$ the term $a(w_s(x))$ in the product in formula \eqref{eqprodmoment} is replaced by $\langle a(w_{s}(x))\theta_i,\theta_i \rangle$} \eqref{eqprodmoment2}.

\end{proof}

\section{Proof of theorem \ref{thmbirkhoff}: tightness.}\label{sectbirkhofftens}

In order to complete the proof of theorem \ref{thmbirkhoff} we will prove the tightness of $(v_t(x,\cdot))_{t\in [0,T]}$ thanks to the following Lemma \ref{encadprod}.

\begin{lem}\label{encadprod}
Let $p \in \N^*$, and $h_1,\dots,h_p :\mathbb{R}_+ \times M \rightarrow \mathbb{R}$ maps such that for any $i=1,\dots, p$, 
\begin{itemize}
\item $\int_{M} h_i(s,\cdot)d\mu =0$ for all $s \geq 0$,
\item $\sup_{s \geq 0}\sum_{a \in \Z}(1+|a|)^{2(1+\e)} \|h_i(s,.,a)\|_{\infty} <\infty$,
\item $h_i(s,.,a)$ is Lipschitz on $\Mp$ and 
$\sup_{s \geq 0}\sum_{a \in \Z}(1+|a|)^{2(1+\e)} \|h_i(s,.,a)\|_{lip} <\infty.$ 
\end{itemize}
Then
\begin{align}\label{eqlemes}
\int_{\frac a \epsilon \leq u_1 \leq \dots \leq u_p \leq \frac b \epsilon }\left|E_{\mup}\left(\prod_{i=1}^ph_i(u_i,T^{\lfloor u_i \rfloor})  \right)\right|du_1\dots du_p=O\left( \left(\frac{b-a}{\epsilon}\right)^{\frac 1 2 \lfloor \frac p 2 \rfloor}\right).
\end{align}
 
\end{lem}

\begin{proof}
Let $k_\epsilon:=\lceil \ln\left( \frac{b-a}{\epsilon} \right)^2\rceil$ and $h_{i,k_\epsilon}(x,.,a):\pM\rightarrow \mathbb R$ an uniform approximation from  formula \eqref{unif} page \pageref{unif} of a map $h_i(x,.,a)$ satisfying the Hypothesis \ref{H2}. According to Lemma \ref{aproxpro}, we will prove the asymptotic behaviour \eqref{eqlemes} by estimating the following term
$$
\epsilon^{\frac 1 2 \lfloor \frac p 2 \rfloor}\int_{\frac a \epsilon \leq u_1 \leq \dots \leq u_p \leq \frac b \epsilon }\left|E_{\mup}\left(\prod_{i=1}^ph_{i,k_\epsilon}(u_i,T^{\lfloor u_i \rfloor})  \right)\right|du_1\dots du_p.
$$ 
This term may be bounded as follows,
\begin{align*}
&\epsilon^{\frac 1 2 \lfloor \frac p 2 \rfloor}\int_{\frac a \epsilon \leq u_1 \leq \dots \leq u_p \leq \frac b \epsilon }\left|E_{\mup}\left(\prod_{i=1}^ph_{i,k_\epsilon}(u_i,T^{\lfloor u_i \rfloor})  \right)\right|d\textbf{u}\leq \epsilon^{\frac 1 2 \lfloor \frac p 2 \rfloor}\sum_{\lfloor \frac a \epsilon \rfloor \leq n_1 \leq \dots \leq n_p \leq \lfloor\frac b \epsilon \rfloor} \sup_{\textbf{u}}\left|E_{\mup}\left(\prod_{i=1}^ph_{i,k_\epsilon}(u_i,T^{n_i})  \right)\right|,\\
\end{align*}
We re-use the notation $b_{a_0,m,n,\textbf{l},\Ng}^{\textbf{e}}(\cdot)$ of section \ref{secbirk} where we substitute \\ $g_{k_\epsilon}(w_{\epsilon (\ls_p+m)}(x),.,a_p)^{N_p}$ by $h_{i,k_\epsilon}(x,.,a_q)$, i.e for $h\in \B$ :
\begin{align*}
&b_{a_0,m,n,\textbf{l}}^{\textbf{e}}(h):=\sup_{\textbf{u}\in \Z^p}\left|\sum_{\textbf{a}\in \Z^p}E_{\pmu}\left(P^{2k_\epsilon}\left(  (h_{p,k_\epsilon}(u_p,.,a_p)H_{k_\epsilon,l_q,a_q-a_{p-1}}^{(e_p)})\circ
\dots \circ (h_{1,k_\epsilon}(u_1,.,a_1)H_{k_\epsilon,l_1,a_1-a_0}^{(e_1)}) (h) \right) \right)\right|.
\end{align*}
Relations \eqref{unpourotus} and
\eqref{erreurptk} still apply, thus the only \textbf{admissible} vectors $\textbf{e}$ are those satisfying relation \eqref{r1}. Thus  
an element $\textbf{e}$ satisfies $e_{2i}=1$ and $e_{2i-1}=0$ for any $1\leq i\leq p$ (otherwise 
 $b_{a_0,m,n,\textbf{l},\Ng}^{\textbf{e}}(h):=\sup_{u_1,\dots,u_p \in \Z}\left|\sum_{a_1,\dots,a_p \in \Z}E_{\pmu}\left(P^{2k_\epsilon}  h_{p,k_\epsilon}(u_p,.,a_p)\right)E_{\mup}(\cdot)\right|=0$). Thus when introducing the notation $\textbf{l}(\textbf u)$ for the vector with coordinates $l_i(\textbf{u})=\lfloor u_i \rfloor-\lfloor u_{i-1} \rfloor$,

\begin{align}
\lim_{\epsilon \rightarrow 0}&\epsilon^{\frac 1 2 \lfloor \frac {2p} 2 \rfloor}\int_{\frac a \epsilon \leq u_1 \leq \dots \leq u_{2p} \leq \frac b \epsilon }\left|E_{\mup}\left(\prod_{i=1}^{2p}h_{i,k_\epsilon}(u_i,T^{\lfloor u_i \rfloor})  \right)\right|d\textbf{u}\nonumber\\
&\leq \lim_{\epsilon \rightarrow 0}\epsilon^{\frac 1 2 \lfloor  p  \rfloor}\int_{\frac a \epsilon \leq u_1 \leq \dots \leq u_{2p} \leq \frac b \epsilon }\prod_{i=1}^pb_{0,0,\epsilon,\textbf{l}(\textbf{u})}^{\textbf{e}}(1)d\textbf{u}\nonumber\\
&\leq \lim_{\epsilon \rightarrow 0}\epsilon^{\frac 1 2 \lfloor  p  \rfloor}\int_{\frac a \epsilon \leq u_1 \leq \dots \leq u_{2p} \leq \frac b \epsilon }\prod_{i=1}^p \frac{1}{2\pi \Sigma (\lfloor u_{2i-1} \rfloor-\lfloor u_{2(i-1)} \rfloor)^{1/2}}\nonumber\\ 
&\sup_{t_{2i},t_{2i-1} \in \mathbb{R}}\left|\sum_{a\in \Z}\hat \Xi_{\epsilon,2,(0,\lfloor u_{2i} \rfloor-\lfloor u_{2i-1} \rfloor),(1,1)}(h_{i,k_\epsilon},(t_{2i-1},t_{2i}),(0,a))\right|d\textbf{u}\nonumber\\
&\leq \lim_{\epsilon \rightarrow 0}\frac{1}{\left(2\pi \Sigma \right)^{1/2}}\sum_{l \in \Z}\sup_{t_{2i},t_{2i-1} \in \mathbb{R}}\left|\hat \Xi_{\epsilon,2,(0,l),(1,1)}(h_{i,k_\epsilon},(t_{2i-1},t_{2i}),(0,a))\right|\label{somconv}\\
&\epsilon^{\frac 1 2 \lfloor  p  \rfloor}\int_{\frac a \epsilon \leq s_1 \leq \frac b \epsilon}\int_{0 \leq s_3 \leq \dots \leq s_{2p-1} \leq \frac {b-a} \epsilon}\prod_{i=1}^p \frac{1}{s_{2i-1}^{1/2}}d\textbf{u},\label{intconv}
\end{align}
where $\hat \Xi_{\epsilon,2,(0,l),(1,1)}(h_{i,k_\epsilon},(t_{2i-1},t_{2i}),(0,a))$ is a notation corresponding to the following quantity
$\hat \Xi_{\epsilon,2,(0,l),(1,1)}(h_{i,k_\epsilon},(t_{2i-1},t_{2i}),(0,a))=E_{\pmu}\left(\left( \prod_{i=1}^q g_{k_\epsilon}(z_i,\pt^{l_i}(\cdot),a_i)^{N_i}1_{\{S_{l_i}=a_i-a_{i-1}\}}\circ \pt^{k_\epsilon} \right)\circ \pt^{\ls_{i-1}} \right)$.
As done in equation \eqref{2eqconvsommoment}, the sum $\sum_{l \in \Z}\sup_{t_{2i},t_{2i-1} \in \mathbb{R}}\left|E_{\mup}\left(h_{i,k_\epsilon}(t_{2i},T^{|l|}\cdot)h_{i,k_\epsilon}(t_{2i-1},\cdot)  \right)\right|$ converges,
thus according to Lemma \ref{convv}, the sum appearing in \eqref{somconv} converges, the asymptotic behavior of formula \eqref{eqlemes} is thus given by the asymptotic behavior from formula \eqref{intconv}. To conclude, we compute the integral in formula \eqref{intconv} using the following estimates: in one hand
$$
\int_{0}^{\frac {b-a} \epsilon}\frac 1 {u^{1/2}} du=\sqrt{\frac{b-a}{\epsilon}}
$$
and on the other hand, we use inequality $b^2-a^2\leq (b-a)^2$ for $b\geq a\geq 0$, i.e $\sqrt{b^2-a^2}\geq (b-a)=\sqrt{b^2}-\sqrt{a^2}$ :
$$
\int_{\frac a \epsilon}^{\frac b \epsilon}\frac 1 {u^{1/2}} du=\sqrt{\frac{b}{\epsilon}}-\sqrt{\frac{a}{\epsilon}}\leq \sqrt{\frac{b-a}{\epsilon}}.
$$
Thus we obtain formula \eqref{eqlemes} proving the lemma.
\end{proof}

\begin{prop}\label{tensvt}
For any $S>0$, the family of processes $\left(\left(  v_t^\epsilon \right)_{t\in [0,T]}\right)_{\epsilon >0}$ is tight in $D([0,T])$ with metric $\|.\|_\infty$. Furthermore\footnote{When $d\geq 2$ it is enough to prove tightness of  $(\langle \theta,v_t^\epsilon \rangle)_{t\geq 0}$ for any $\theta \in \mathbb R^d$. The proof  follows the same path replacing $(v_t^\epsilon)_{t\geq 0}$ by  $(\langle \theta,v_t^\epsilon \rangle)_{t\geq 0}$.}, there is some constant $K>0$ such that for any $p\in \N^*$ ,
\begin{align*}
    \| v_t^\epsilon -  v_s^\epsilon \|_{L^{2p}_{\mup}}\leq K \left(t-s\right)^{\frac 1 4}
\end{align*}
\end{prop}
\begin{proof}
According to Billingsley Theorem \cite[theorem 13.5]{billingsley2}, it is enough to show that $\beta \geq 0$, $\alpha >\frac 1 2$ and $C>0$, such that for any $0 \leq r \leq s \leq t \leq T$,
\begin{align}\label{prouveca}
E_{\mup}\left( |\tilde v_s^\epsilon - \tilde v_r^\epsilon|^{2\beta} |\tilde v_t^\epsilon - \tilde v_s^\epsilon|^{2\beta} \right) \leq C(t-r)^{2\alpha}.
\end{align} 

Then  Cauchy-Schwarz inequality gives, 
\begin{align}\label{olderun}
E_{\mup}\left( | v_s^\epsilon -  v_r^\epsilon|^{2\beta} | v_t^\epsilon -  v_s^\epsilon|^{2\beta} \right) \leq \| v_s^\epsilon - v_r^\epsilon \|_{L^{4\beta}_{\mup}}^{2\beta}\| v_t^\epsilon -  v_s^\epsilon \|_{L^{4\beta}_{\mup}}^{2\beta}.
\end{align}
Then we apply Lemma \ref{encadprod} for $p \in \N^*$ on $h_i(s,\cdot)=\F(w_{\epsilon s}(x),\cdot)$,
\begin{align}
\| v_t^\epsilon -  v_s^\epsilon \|_{L^{2p}_{\mup}}&=\left\|\epsilon^{1/4}\int_{\frac s \epsilon}^{\frac t \epsilon}\F(w_{\epsilon u}(x),T^{\lfloor u \rfloor}(\cdot))du \right\|_{L^{2p}_{\mup}}\nonumber\\
&=\epsilon^{1/4}\left(p!\int_{\frac s \epsilon \leq u_1 \leq \dots \leq u_{2p} \leq \frac t \epsilon }\left|E\left(\prod_{i=1}^p\F(w_{\epsilon u_i}(x),T^{\lfloor u_i \rfloor}(\cdot))\right)\right|du_1 \dots du_p\right)\nonumber\\
&\leq K \left(t-s\right)^{\frac 1 4}.\label{lptens}
\end{align}
Then we introduce this inequality in relation \eqref{olderun} and we conclude the tightness by checking the assumption from \cite[theorem 13.5]{billingsley2} with parameters $\beta=2$, $\alpha=\frac{\beta}{2}>\frac 1 2$ by choosing $p=2\beta$  :
\begin{align*}
E_{\mup}\left( | v_s^\epsilon -  v_r^\epsilon|^{2\beta} | v_t^\epsilon -  v_s^\epsilon|^{2\beta} \right) &\leq K(s-r)^{\frac{4\beta}{8}}(t-s)^{\frac{4\beta}{8}}\\
&\leq K (t-r)^{2\alpha}.
\end{align*}
\end{proof}

\section{Averaging : Proof of theorem \ref{thmequadif}.}\label{secpreuvedif}
The proof of the main theorem \ref{thmequadif} relies on the statements of theorem \ref{thmbirkhoff} and Lemma \ref{encadprod}.
To control the error term $e_t^{\epsilon}$, we will give a variational expression of its that we will control using classical Grönwall lemma and the behavior of the perturbed sum $v_t^\epsilon$ that we recall here
$$
v_t^\epsilon(x,\omega):=\epsilon^{1/4}\int_0^{t/\epsilon} \F(w_{\epsilon s}(x),T^{\lfloor s \rfloor}\omega)ds,
$$
for $x \in \mathbb R^d$, $\omega \in M$ and $t \in [0,S]$.

\begin{proof}[Proof of theorem \ref{thmequadif}]
 Through this proof, we keep the notations from theorem \ref{thmequadif} and we first provide a loose control on $(e_t^\epsilon)_{t\geq 0}$ with the following Lemma \ref{errlp}.

\begin{lem}\label{errlp}
For any $S>0$ and $p \geq 1$, the error term $e_t^\epsilon(x,\cdot)=x_t^\epsilon(x,\cdot)-w_t^\epsilon(x,\cdot)$ satisfies,
\begin{align}\label{trucunun}
\sup_{t \in [0,T']}\|\epsilon^{-3/4}e_t^\epsilon(x,\cdot)\|_{L^p_{\mup}}\underset{\epsilon \rightarrow 0}=O(1),
\end{align}
and the family $\left( \left(\epsilon^{-3/4}e_t^\epsilon(x,\cdot)\right)_{t\geq 0}\right)_{\epsilon >0}$ is tight in $C^0([0,T'],\mathbb{R})$.
\end{lem}

\begin{proof}
We first prove relation \eqref{trucunun}. We recall that\\ $\overline{F}$, $(x_t^\epsilon(x,\cdot))_{t \in [0,T']}$ and $(w_t(x))_{t \in [0,T']}$ satisfy respectively relations \eqref{eqperturbdifftf} and \eqref{eqomega}. For any $\omega \in M$,
\begin{align}\label{fromulerere}
e_t^\epsilon(x,\omega)&=x_t^\epsilon(x,\omega)-w_t^\epsilon(x,\cdot)\nonumber\\
&=\epsilon^{3/4}v_t^\epsilon(x,\omega)+\int_0^tF\left(x_s^\epsilon(x,\omega),T^{\lfloor \frac s \epsilon \rfloor}(\omega)\right)-F\left(w_s(x),T^{\lfloor \frac s \epsilon \rfloor}(\omega)\right)ds \nonumber\\
&+\int_0^t\overline F(x_s^\epsilon(x,\omega))-\overline F(w_s(x))ds.
\end{align}
We deduce from equation \eqref{fromulerere} and the inequality of finite increased,
$$
\left|\epsilon^{-3/4}e_t^\epsilon(x,w)\right|\leq |v_t^\epsilon(x,w)|+\int_0^t\left(\|D_1F\|_{\infty}+\|D_1\overline F\|_{\infty}\right)\left|\epsilon^{-3/4}e_s^\epsilon(x,w)\right|ds.
$$
Then we apply Grönwall inequality and take the $L^p_{\mup}$ norm, 
$$
\left|\epsilon^{-3/4}e_t^\epsilon(x,\omega)\right|\leq |v_t^\epsilon(x,\omega)|e^{T'\left(\|D_1F\|_{\infty}+\|D_1\overline F\|_{\infty}\right)}.
$$
Thus for $p\in \N^*$,
$$
\sup_{t \in [0,T]}\|\epsilon^{-3/4}e_t^\epsilon(x,\cdot)\|_{L^p_{\mup}}\leq \sup_{t \in [0,T']}\| v_t^\epsilon(x,\cdot)\|_{L^p_{\mup}}e^{T'\left(\|D_1F\|_{\infty}+\|D_1\overline F\|_{\infty}\right)}
$$
According to Lemma \ref{encadprod}, $\| v_t^\epsilon(x,\cdot)\|_{L^p_{\mup}}$ is uniformly bounded in $\epsilon$, thus we deduce \eqref{trucunun}. \\
To prove the tightness, 
it is enough to prove that for any $p \in \N^*$ there is $C>0$ such that 
\begin{align}\label{eqbiles}
\sup_{0\leq s\leq t \leq T'}\|\epsilon^{-3/4}(e_t^\epsilon(x,\cdot)-e_s^\epsilon(x,\cdot))\|_{L^p_{\mup}}\leq C(t-s)^{\frac 1 4}.
\end{align}
If we suppose relation  \eqref{eqbiles} holds,  we can then apply Cauchy-Schwarz formula with $p=2\beta$ to obtain the following relation
\begin{align*}
E_{\mup}\left( | e_s^\epsilon -  e_r^\epsilon|^{2\beta} | e_t^\epsilon -  e_s^\epsilon|^{2\beta} \right) &\leq K(s-r)^{\frac{4\beta}{8}}(t-s)^{\frac{4\beta}{8}}\\
&\leq K (t-r)^{2\alpha}.
\end{align*}
This relation corresponds to the tightness condition of theorem 15.6 from \cite[theorem 13.5]{billingsley2} when  choosing $\beta=2$ and $\alpha=\frac{\beta}{2}>\frac 1 2$ thus proving the tightness of $(e_t^\epsilon)_{t\geq 0}$ in Lemma \ref{errlp}.

To prove relation \eqref{eqbiles}, we fix $0\leq s \leq t \leq T'$ and obtain the following relation by applying the mean value theorem within formula \eqref{fromulerere}  :
$$
\epsilon^{-3/4}|e_t^\epsilon(x,\cdot)-e_s^\epsilon(x,\cdot)| \leq | v_t^\epsilon(x,\cdot)- v_s^\epsilon(x,\cdot)|+\int_s^t\left(\|D_1F\|_{\infty}+\|D_1\overline F\|_{\infty}\right)\left|\frac{e^\epsilon_u(x,\omega)}{\epsilon^{3/4}}\right|du.
$$
Then we conclude from Lemma \ref{encadprod} and the estimate \eqref{tensvt} from theorem \ref{thmbirkhoff} that for any $p\in \N^*$,
\begin{align*}
\epsilon^{-3/4}\|e_t^\epsilon(x,\cdot)-e_s^\epsilon(x,\cdot)\|_{L^{2p}_{\mup}}&\leq \| v_t^\epsilon(x,\cdot)- v_s^\epsilon(x,\cdot)\|_{L^{2p}_{\mup}}\\
&+(t-s)\left(\|D_1F\|_{\infty}+\|D_1\overline F\|_{\infty}\right)\sup_{u \in [0,T']}\left|\frac{e^\epsilon_u(x,\omega)}{\epsilon^{3/4}}\right|\\
&\leq C(t-s)^{\frac 1 4}.
\end{align*}

\end{proof}
We will now compare through variational analysis the error term $(e_t^\epsilon)_{t\geq 0}$ with the solution
$(y_t^\epsilon(x,\omega))_{t\geq 0}$ of the following differential equation :
\begin{align}\label{equdiffun}
y_t^\epsilon(x,\omega)= v_t^\epsilon(x,\omega)+\int_0^t D\overline F(w_s(x)).y_s^\epsilon(x,\omega)ds.
\end{align}
An explicit expression for $(y_t^\epsilon(x,\omega))_{t\geq 0}$ is given by 
$\left(y_t^\epsilon(x,\omega)\right)_{t \in [0,T']}=\mathcal{F}\left(\left( v_t^\epsilon(x,\omega)\right)_{t \in [0,T']}\right)$, where $\mathcal{F}$ is a bounded linear operator on $C^0([0,T'],\mathbb{R}^d)$ defined as
\begin{align*}
\mathcal{F}((v_t)_{t\geq 0}):=&\left( v_t +\int_0^t D\overline F(w_s(x))v_s\exp\left(\int_s^tD\overline F(w_u(x))du\right)ds\right)_{t \in [0,S]}.\\
\end{align*}

 The convergence in law $\mathcal L_{\mup}$ on the space $(C^0([0,T'],\mathbb{R}^d),\|.\|_\infty)$ of $\left( v_t^\epsilon(x,\cdot)\right)_{t\in [0,T']}$ to the process $\left(\int_0^t \sqrt{a(w_s(x))}dB_{L_s'(0)}\right)_{t\geq 0}$ defined in theorem \ref{thmbirkhoff} implies that
\begin{align}\label{yconv}
\left( v_t^\epsilon(x,\cdot),y_t^\epsilon(x,\cdot)\right)_{t\in [0,T']} \overset{\mathcal{L}_{\mup},\|.\|_\infty}{\underset {\epsilon \rightarrow 0}{\rightarrow}}(\int_0^t \sqrt{a(w_s(x))}dB_{L_s'(0)},y_t(x,\cdot))_{t \in [0,T']}.
\end{align}

We then deduce from Lemma \ref{lemerorety} that follows the convergence of $(e_t^\epsilon)_{t\geq 0}$ for the law $\mathcal L_{\mup}$ to the same process $(y_t)_{t\geq 0}$.

\begin{lem}\label{lemerorety}
Under the assumptions of theorem \ref{thmequadif},
for any $p \geq 1$,
$$
\sup_{t \in [0,T']}\|\epsilon^{-3/4}e_t^\epsilon(x,\cdot)-y_t^\epsilon(x,\cdot)\|_{L^p_{\mup}}\underset{\epsilon \rightarrow 0}{=}O(\epsilon^{3/4}).
$$
\end{lem}
\begin{proof}
We introduce 
$$
a^\epsilon_t(x,\omega):=\epsilon^{-3/4}e_t^\epsilon(x,\omega)- v_t^\epsilon(x,\omega)-\epsilon^{-3/4}\int_0^tD_1F(w_s(x),T^{\lfloor \frac s \epsilon\rfloor}(\omega)).e_s^\epsilon(x,\omega)ds.\\
$$
According to equation \eqref{fromulerere} and some Taylor expansion,
\begin{align*}
|a^\epsilon_t(x,\omega)|&=\epsilon^{-3/4}\left|\int_0^t\left(F'\left(x_s^\epsilon(x,\omega),T^{\lfloor \frac s \epsilon \rfloor}(\omega)\right)\right.\right.\\
&\left.\left.-F'\left(w_s(x),T^{\lfloor \frac s \epsilon \rfloor}(\omega)\right) \right) -D_1F'(w_s(x),T^{\lfloor \frac s \epsilon\rfloor}(\omega)).e_s^\epsilon(x,\omega)ds\right|\\
&\leq \epsilon^{-3/4}\left|\int_0^T\|D_1^2F'\|_{\infty}\left(e_s^\epsilon(x,\omega)\right)^2ds\right|\\
&\leq  \epsilon^{3/4}\left|\int_0^T\|D_1^2F'\|_{\infty}\left(\frac{e_s^\epsilon(x,\omega)}{\epsilon^{3/4}}\right)^2ds\right|.
\end{align*}
Thus, according to Lemma \ref{errlp}.
\begin{align}\label{bornep2}
\sup_{t \in [0,T]}&\left\|a_t^\epsilon(x,\omega)\right\|_{L^p_{\mup}}\leq \epsilon^{3/4}\|D_1^2F'\|_{\infty}T'\sup_{t \in [0,T']}\left\|\frac{e_t^\epsilon(x,\cdot)}{\epsilon^{3/4}}\right\|_{L^{2p}_{\mup}}^2=O\left(\epsilon^{3/4}\right).
\end{align}

We now introduce the process $(b_t^\epsilon(x,\cdot))_{t\geq 0}$ as follows,
\begin{align*}
b_t^\epsilon(x,\omega)&:=y_t^\epsilon(x,\omega)- v_t^\epsilon(x,\omega)-\int_0^t D_1 F'(w_s(x),T^{\lfloor \frac s \epsilon \lfloor}(\omega)).y_s^\epsilon(x,\omega)ds\\
&=\int_0^tD_1 \F(w_s(x),T^{\lfloor \frac s \epsilon \lfloor}(\omega))\left[\epsilon^{-3/4}\int_0^s\F(w_u(x),T^{\lfloor \frac u \epsilon \rfloor}(\omega))du\right. \\
&\left.+\int_0^sD\overline F(w_b(x))\exp\left(\int_b^sD\overline F(w_a(x))da\right)\int_0^b\F(w_u(x),T^{\lfloor \frac u \epsilon \rfloor}(\omega))du db\right]ds\\
&=\epsilon^{-\frac 3 4}\int_{0 \leq u \leq s \leq t}K_{u,s}D_1 \F(w_s(x),T^{\lfloor \frac s \epsilon \rfloor}(\omega)).\F(w_u(x),T^{\lfloor \frac u \epsilon \rfloor}(\omega))duds,
\end{align*}
where for any $u \leq s$, the quantity
$$
K_{u,s}:=1+\int_u^sD\overline F(w_b(x))\exp\left(\int_b^sD\overline F(w_a(x))da\right)db
$$
satisfies $\|K_{u,s}\|_\infty\leq 1+T'\|D\overline{F}\|_\infty e^{T\|D\overline{F}\|_\infty}=:K_0$.
Thus Lemma \ref{errlp} applies and we can conclude that for any $p \in \N^*$,
\begin{align*}
\|b_t^\epsilon\|_{L^{2p}_{\mup}}^{2p}&\leq K_0'^{2p} \epsilon^{-3p/2}\int_{0\leq u_1 \leq s_1 \leq t}\dots \int_{0\leq u_1 \leq s_1 \leq t}\\
& \left|E_{\mup}\left(\prod_{i=1}^{2p}D_1 \F(w_s(x),T^{\lfloor \frac s \epsilon \rfloor}(\omega)).\F(w_u(x),T^{\lfloor \frac u \epsilon \rfloor}(\omega))\right)\right|du_1ds_1\dots du_{2p}ds_{2p}\\
&=O\left(\epsilon^{-\frac{3p}{2}}\epsilon^{4p}\left(\frac{t}{\epsilon}\right)^{2p}\right)=O(S^p\epsilon^{\frac{3p}{2}}).
\end{align*}
Thus,
\begin{align}\label{bornb}
\sup_{t \in [0,T']}\|b_t^\epsilon(x,\cdot)\|_{L^p_{\mup}}=O(\epsilon^{\frac 3 4}).
\end{align}
Now we apply Grönwall lemma on the following variational equation
\begin{align*}
&\epsilon^{-3/4}e_t^\epsilon(x,\omega)-y_t^\epsilon(x,\omega)=a_t^\epsilon(x,\omega)-b_t^\epsilon(x,\omega)\\
&+\int_0^tD_1F'(w_s(x),T^{\lfloor \frac s \epsilon \rfloor}(\omega))(\epsilon^{-\frac 3 4}e_s^\epsilon(x,\omega)-y_s^\epsilon(x,\omega))ds,
\end{align*}
and we obtain the following estimate
\begin{align*}
\left|\epsilon^{-3/4}e_t^\epsilon(x,\omega)-y_t^\epsilon(x,\omega)\right| \leq \left|a_t^\epsilon(x,\omega)-b_t^\epsilon(x,\omega)\right|e^{T'\|D_1F\|_{\infty}}.
\end{align*}
Then estimates \eqref{bornb} and \eqref{bornep2} provide the conclusion \eqref{lemerorety} of the lemma,
$$
\sup_{t \in [0,T']}\|\epsilon^{-3/4}e_t^\epsilon(x,\omega)-y_t^\epsilon(x,\omega)\|_{L^{2p}_{\mup}}=\sup_{t\in [0,T']}(\|a_t^\epsilon(x,\cdot)\|_{L^{2p}_{\mup}}+\|b_t^\epsilon(x,\cdot)\|_{L^{2p}_{\mup}}=O(\epsilon^{\frac 3 4}).
$$
\end{proof}

From Lemma \ref{lemerorety} above and the convergence in of $(y_t^\epsilon)_{t\geq 0}$ in \eqref{yconv}, we deduce the conclusion of theorem \ref{thmequadif} in the particular case where the probability law $\PP$ in the theorem is precisely the law $\mup$ on $\Mp\simeq M\times \{0\}$ :
\begin{align}\label{eqthmmup}
(\epsilon^{-3/4}e_t^\epsilon(x,\omega))_{t \in [0,T']}\underset{\epsilon \rightarrow 0}{\overset{\mathcal{L}_{\mup},\|.\|_{\infty}}{\rightarrow}}(y_t)_{t \in [0,T']}.    
\end{align}
To prove the full theorem and get the conclusion for any probability measure $\PP$ absolutely continuous with respect to $\mu$, it is enough to prove that the assumptions from Zweimuller theorem \cite[Theorem 1]{zweimuller} hold. In our case, those assumptions can be summarized as the following condition :

\begin{itemize}
    \item \textbf{Condition} : for any $x\in \mathbb R ^d$ and $\omega \in M$,\\
    \begin{align}\label{eqcondzwei}
       \epsilon^{-3/4} \left\| \left(e_t^\epsilon(x,\omega)-e_t^\epsilon(x,T\omega)\right)_{t\geq 0} \right\|_\infty \xrightarrow[\epsilon \to 0]{}0.
    \end{align}
\end{itemize}

We introduce the notation 
\begin{align}\label{zweieq2}
u_t^{\epsilon}(x,\omega):=\frac{e_t^\epsilon(x_0,T(\omega))-e_t^\epsilon(x_0,\omega)}{\epsilon^{3/4}}=\frac{x_t^\epsilon(x_0,T(\omega))-x_t^\epsilon(x_0,\omega)}{\epsilon^{3/4}}.
\end{align}

From the perturbed differential equation \ref{pertx}, $(x_t^\epsilon)_{t\geq 0}$ satisfies the following variational inequality 
\begin{align*}
&|x_s^{\epsilon}(x_0,T(\omega))-x_s^{\epsilon}(x_\epsilon^\epsilon(x_0,\omega),T(\omega))|\leq |x_0-x_\epsilon^\epsilon(x_0,\omega)|\\
&+\int_0^s[F]|x_u^{\epsilon}(x_0,T(\omega))-x_u^{\epsilon}(x_\epsilon^\epsilon(x_0,\omega),T(\omega))|du,
\end{align*}
and thus applying Grönwall inequality on $|x_s^{\epsilon}(x_0,T(\omega))-x_s^{\epsilon}(x_\epsilon^\epsilon(x_0,\omega),T(\omega))|$, we deduce the following estimate
\begin{align}\label{eq:xs}
\sup_{s\in [0,S]}|x_s^{\epsilon}(x_0,T(\omega))-x_s^{\epsilon}(x_\epsilon^\epsilon(x_0,\omega),T(\omega))|=O(|x_0-x_\epsilon^\epsilon(x_0,\omega)|)=O(\epsilon).    
\end{align}

The \textbf{condition} given by equation \eqref{eqcondzwei} is then a consequence of the following estimate of $u_t^\epsilon(x,\omega)$ using the variational formulation of $(x_t^\epsilon)_{t\geq 0}$ and the estimate \eqref{eq:xs} above :

\begin{align*}
|u_t^{\epsilon}(x,\omega)|&=\left|\frac{x_t^\epsilon(x_0,T(\omega))-x_{t-\epsilon}^\epsilon(x_\epsilon^\epsilon(x_0,\omega),T(\omega))}{\epsilon^{3/4}}\right|\\
&\leq \epsilon^{-1/2}\left|\int_{t-\epsilon}^t F(x_s^{\epsilon}(x,T(\omega),T^{\lfloor \frac{s}{\epsilon}\rfloor}(T\omega))ds\right|\\
&+\epsilon^{-1/2}\int_0^{t-\epsilon}[F]|x_s^{\epsilon}(x_0,T(\omega))-x_s^{\epsilon}(x_\epsilon^\epsilon(x_0,\omega),T(\omega))|ds\\
&=O(\epsilon^{1/4}).
\end{align*}
Thus theorem \cite[Theorem 1]{zweimuller} applies and the convergence in law in \eqref{eqthmmup}holds for any measure $\PP$ absolutely continuous w.r to $\mu$ :
$$
(\epsilon^{-3/4}e_t^\epsilon(x,\omega))_{t \in [0,T']}\underset{\epsilon \rightarrow 0}{\overset{\mathcal{L}_{\PP},\|.\|_{\infty}}{\rightarrow}}(y_t)_{t \in [0,T']}.
$$ This ends the proof of theorem \ref{thmequadif}.
\end{proof}
\subsection*{Aknowledgement} This article complete some work initiated during my PhD thesis under the supervision of Françoise Pène whom I want to thank for leading me into this study providing precious advises as well as a strong editorial support.

\bibliography{biblio2}
\bibliographystyle{plain}
\end{document}